\documentclass[11pt,a4paper]{amsart}

\usepackage{graphicx}
\usepackage{xcolor}
\usepackage{amssymb}
\usepackage{physics}
\usepackage{mathtools}
\usepackage{booktabs}

\usepackage[
    margin=25mm,
    bottom=30mm
]{geometry}

\usepackage{hyperref}
\usepackage{orcidlink}

\hypersetup{
    hidelinks,
    bookmarksnumbered=true,
    pdftitle={Parameter Estimation for Diffusive Stochastic Master Equations in Continuously Observed Quantum Systems},
    pdfauthor={Mitsuki Kobayashi, Shohei Nakajima},
    pdfkeywords={parameter estimation, stochastic master equations, homodyne detection, consistency, asymptotic normality},
}

\theoremstyle{plain}
\newtheorem{theorem}{Theorem}[section]
\newtheorem{proposition}[theorem]{Proposition}
\newtheorem{lemma}[theorem]{Lemma}

\theoremstyle{definition}
\newtheorem{assumption}{Assumption}

\theoremstyle{remark}
\newtheorem*{remark}{Remark}
\newtheorem{numericalremark}{Remark}[section]

\mathtoolsset{showonlyrefs,showmanualtags}
\numberwithin{equation}{section}
\usepackage{zref-clever}
\zcsetup{S,nameinlink=true}

\AddToHook{env/proposition/begin}{\zcsetup{countertype={theorem=proposition}}}
\AddToHook{env/lemma/begin}{\zcsetup{countertype={theorem=lemma}}}
\AddToHook{env/corollary/begin}{\zcsetup{countertype={theorem=corollary}}}

\zcRefTypeSetup{assumption}{
    Name-sg=Assumption,
    Name-pl=Assumptions,
    name-sg=assumption,
    name-pl=assumptions
}
\zcRefTypeSetup{theorem}{
    Name-sg=Theorem,
    Name-pl=Theorems,
    name-sg=theorem,
    name-pl=theorems
}
\zcRefTypeSetup{proposition}{
    Name-sg=Proposition,
    Name-pl=Propositions,
    name-sg=proposition,
    name-pl=propositions
}
\zcRefTypeSetup{lemma}{
    Name-sg=Lemma,
    Name-pl=Lemmas,
    name-sg=lemma,
    name-pl=lemmas
}
\zcRefTypeSetup{corollary}{
    Name-sg=Corollary,
    Name-pl=Corollaries,
    name-sg=corollary,
    name-pl=corollaries
}
\zcRefTypeSetup{numericalremark}{
    Name-sg=Remark,
    Name-pl=Remarks,
    name-sg=remark,
    name-pl=remarks
}

\usepackage{etoolbox}

\newcommand{\Real}{\operatorname{Re}}

\newcommand{\Lvec}{\mathbf{L}}

\NewDocumentCommand{\Wvec}{o}{%
    \mathbf{W}%
    \IfValueT{#1}{^{(#1)}}%
}

\NewDocumentCommand{\Yvec}{o}{%
    \mathbf{Y}%
    \IfValueT{#1}{^{(#1)}}%
}

\newcommand{\Wbar}{\overline{\mathbf{W}}^N}
\newcommand{\Ybar}{\overline{\mathbf{Y}}^N}

\DeclarePairedDelimiterX{\innerp}[1]{\langle}{\rangle}{\innpargs{#1}}

\NewDocumentCommand{\innpargs}{>{\SplitArgument{1}{,}}m}{\innpargsaux#1}

\NewDocumentCommand{\innpargsaux}{mm}{%
    \ifblank{#1}{%
        \ifblank{#2}
            {{\,\cdot\,}{,}{\,\cdot\,}}%
            {{\,\cdot\,}{,}{\mkern2mu#2}}%
    }{%
        {{#1\mkern2mu}{,}\ifblank{#2}{\,\cdot\,}{\mkern2mu#2}}%
    }%
}

\NewDocumentCommand{\innerpHS}{s m}{%
    \IfBooleanTF{#1}
        {\innerp*{#2}_{\mathrm{HS}}}
        {\innerp{#2}_{\mathrm{HS}}}%
}

\DeclarePairedDelimiterX{\bformparen}[1]{(}{)}{\innpargs{#1}}

\NewDocumentCommand{\bform}{s m}{%
    b%
    \IfBooleanTF{#1}
        {\bformparen*{#2}}
        {\bformparen{#2}}%
}

\NewDocumentCommand{\rhosol}{>{\SplitArgument{2}{,}}m}{\rhosolaux#1}

\NewDocumentCommand{\rhosolaux}{m m m}{%
    \rho_{#1}^{#2\IfNoValueF{#3}{,(#3)}}%
}

\NewDocumentCommand{\rhobar}{>{\SplitArgument{2}{,}}m}{\rhobaraux#1}

\NewDocumentCommand{\rhobaraux}{m m m}{%
    \bar{\rho}_{#1}^{#2\IfNoValueF{#3}{,#3}}%
}

\NewDocumentCommand{\rhoave}{>{\SplitArgument{1}{,}}m}{\rhoaveaux#1}

\NewDocumentCommand{\rhoaveaux}{m m}{%
    \bar{\rho}_{#1}^{#2}%
}

\DeclareMathOperator*{\argmax}{arg\,max}
\newcommand{\convas}{\xrightarrow{\mathrm{a.s.}}}
\newcommand{\convp}{\xrightarrow{P}}

\begin{document}

\title[Parameter Estimation for Diffusive SMEs]{Parameter Estimation for Diffusive Stochastic Master Equations in Continuously Observed Quantum Systems}

\author[Mitsuki Kobayashi]{Mitsuki Kobayashi\,\orcidlink{0000-0003-0083-3187}}
\address{Faculty of Science and Technology, Seikei University}
\email{mitsuki-kobayashi@st.seikei.ac.jp}

\author[Shohei Nakajima]{Shohei Nakajima\,\orcidlink{0000-0002-0407-7625}}
\address{Faculty of Management, Tokyo University of Science}
\email{snakajima@rs.tus.ac.jp}


\begin{abstract}
    Continuous measurement of quantum systems gives rise to stochastic dynamics of
    the conditional quantum state, described by diffusive stochastic master
    equations. In this paper, we study parameter estimation for such equations
    when the Hamiltonian and measurement operators depend on unknown parameters.
    Based on multiple independent observed trajectories with a known initial state, we construct a
    contrast function using the deterministic averaged state and define a maximum
    contrast estimator for the unknown parameter. We prove strong consistency and
    asymptotic normality of the parameter estimator in a fixed-time,
    many-trajectory asymptotic regime. A key point is that the covariance matrix
    appearing in the asymptotic normality is given in a form that naturally leads
    to a consistent covariance estimator. This covariance estimator is
    computable from the observed data together with the deterministic averaged
    dynamics, so the asymptotic normality result can be used to construct
    standard errors and assess uncertainty for the parameter estimator.
\end{abstract}

\maketitle

\section{Introduction}

We consider the stochastic master equation on the time interval \([0,T]\):
\begin{equation}
    d\rhosol{t,\theta}
    =
    \mathcal{L}^{\theta}\bigl(\rhosol{t,\theta}\bigr)\,dt
    +
    \mathcal{H}^{\beta}\bigl(\rhosol{t,\theta}\bigr)\cdot d\Wvec_t,
    \qquad
    \theta\in\Theta.
    \label{eq:sMasterEq}
\end{equation}
Here, \(\Theta=\overline{\Theta^\circ}\), where \(\Theta^\circ\) is a bounded convex open subset of a finite-dimensional Euclidean space. We write \(\theta=(\alpha,\beta)\), where \(\alpha\) and \(\beta\) are the parameters entering the Hamiltonian and the measurement operators, respectively, and each component is vector-valued. The true value is denoted by \(\theta_0=(\alpha_0,\beta_0)\in\Theta^\circ\). The process \(\rhosol{t,\theta}\) is a density matrix in \(M_d(\mathbb{C})\), and
\(\Wvec=(\Wvec_t)_{0\leq t\leq T}\), with \(\Wvec_t=(W_t^{(1)},\dots,W_t^{(r)})^\top\), is an \(r\)-dimensional standard Wiener process. The drift term is
\begin{equation}
    \mathcal{L}^{\theta}(\rho)
    :=
    -i[H^\alpha,\rho]
    +
    \sum_{k=1}^{r}
    \left(
        L_k^\beta\rho(L_k^\beta)^*
        -
        \frac{1}{2}
        \left\{
            (L_k^\beta)^*L_k^\beta,\rho
        \right\}
    \right),
\end{equation}
where the Hamiltonian \(H^\alpha\in M_d(\mathbb{C})\) is self-adjoint and
\(L_1^\beta,\dots,L_r^\beta\in M_d(\mathbb{C})\) are measurement operators. For \(A,B\in M_d(\mathbb{C})\),
\begin{equation}
    [A,B]
    :=
    AB-BA,
    \qquad
    \{A,B\}
    :=
    AB+BA.
\end{equation}
The stochastic term is interpreted componentwise as
\begin{equation}
    \mathcal{H}^{\beta}(\rho)\cdot d\Wvec_t
    =
    \sum_{k=1}^{r}
    \left(
        L_k^\beta\rho
        +
        \rho(L_k^\beta)^*
        -
        \Tr\!\left(
            L_k^\beta\rho
            +
            \rho(L_k^\beta)^*
        \right)\rho
    \right)dW_t^{(k)}.
\end{equation}

Equations of the form \eqref{eq:sMasterEq} arise naturally in quantum filtering theory for continuously monitored quantum systems \cite{barchielli2009quantum,barchielli1995constructing,belavkin1989new, belavkin1992quantum,bouten2007introduction,wiseman2010quantum}. The deterministic part \(\mathcal{L}^{\theta}\) is of Lindblad type and describes the averaged open-system dynamics generated by the Hamiltonian and the measurement operators. The structure of such generators goes back to the theory of quantum dynamical semigroups \cite{lindblad1976generators}. The stochastic part describes the measurement back-action and is driven by the innovation process associated with the observation record.

More explicitly, if
\begin{equation}
    \Yvec_t
    :=
    (Y_t^{(1)},\dots,Y_t^{(r)})^\top
\end{equation}
denotes the observed measurement record, then its \(k\)-th component is formally written as
\begin{equation}
    dY_t^{(k)}
    =
    \Tr\!\left(
        L_k^\beta\rhosol{t,\theta}
        +
        \rhosol{t,\theta}(L_k^\beta)^*
    \right)dt
    +
    dW_t^{(k)},
    \qquad
    k=1,\dots,r.
    \label{eq:innovation}
\end{equation}
The stochastic master equation \eqref{eq:sMasterEq} describes the evolution of the
conditional state \(\rhosol{t,\theta}\) given the observation record
\((\Yvec_s)_{0\leq s\leq t}\), with the corresponding innovation process as the driving noise.
For the physical background of the unit-noise normalization in
\eqref{eq:innovation}, see, \emph{e.g.}, \cite[Eq.~(5.5)]{bouten2007introduction}.

Identifying quantum mechanical models with unknown parameters from observational data is a fundamental and practical problem in quantum information science. In many experimental settings, precise knowledge of system parameters, such as Hamiltonian parameters, measurement efficiencies, and coupling strengths, is unavailable, although these quantities determine the behavior of quantum devices. Since experimentally realized systems generally interact with their environments, a statistical theory is needed for estimating parameters of open quantum dynamics from actual measurement records.

An early contribution to the statistical identification of continuously monitored open quantum systems is due to Mabuchi \cite{mabuchi1996dynamical}, who studied estimation of Hamiltonian parameters from measurement records. Since then, parameter estimation for continuously monitored quantum systems has often been studied through likelihood, Bayesian, filtering-based, or recursive numerical procedures. See, \emph{e.g.}, \cite{gambetta2001state, gammelmark2013bayesian,kiilerich2016bayesian,ralph2017multiparameter} and the references therein. 

 A characteristic feature of continuous monitoring is that information about the unknown parameter is accumulated over time through the measurement record.  At the same time, the observation equation \eqref{eq:innovation} is not a standard diffusion model with a drift given explicitly as a function of the parameter. Its drift is determined through the conditional state, which is itself obtained from the stochastic master equation driven by the same observed record. Consequently, a full likelihood for the observation record must be evaluated together with the quantum filter, or equivalently with the corresponding unnormalized filtering equation. This feature distinguishes the problem from standard parameter estimation for classical diffusion processes, where the drift coefficient is usually specified directly as a function of the state and the parameter. See, \emph{e.g.}, \cite{kutoyants2004statistical,prakasa1999statistical}.

Motivated by these considerations, the present work focuses on the large-sample statistical properties of a parameter estimator constructed from multiple independent trajectories. We consider a setting in which \(N\) independent realizations of the measurement record are obtained under the same experimental configuration, where the initial state \(\rho_0\) is part of the experimental preparation and is therefore known. Although this setting requires repeated experiments and therefore entails an experimental cost, it provides a natural large-sample regime in which consistency, asymptotic normality, and asymptotic covariance can be studied rigorously. A further advantage of the proposed estimator is its computational simplicity. In the full likelihood approach, the drift of the observation process is evaluated through the conditional state, and hence the quantum filtering equation has to be solved along each observed trajectory and for each candidate value of the parameter. This can be computationally demanding, especially when numerical optimization over the parameter space is required. The proposed estimator is constructed by replacing the filter-dependent drift by a quantity computed from the averaged state \(\rhoave{t,\theta}\). See \eqref{eq:mce} below for its precise definition. It is not necessary to solve the stochastic master equation for each measurement record. Instead, one only needs to solve the deterministic master equation for the averaged dynamics. This substantially reduces the computational cost. We also derive a covariance formula in the limiting normal distribution and show that it leads to a consistent empirical covariance estimator computed from the observed trajectories and the deterministic averaged dynamics. Hence the asymptotic theory goes beyond a formal normal approximation: it provides standard errors and a componentwise Studentized statistic for uncertainty quantification. We also consider estimation from discretely sampled measurement records and establish the corresponding asymptotic properties as the sampling mesh is refined.

The rest of the paper is organized as follows. Section~\ref{sec:preliminaries} collects the notation and assumptions used throughout the paper. 
Section~\ref{sec:main-results} presents our main results.
Section~\ref{sec:proof-main-results} is devoted to the proofs. The proof begins with the differential structure of the contrast function and several auxiliary convergence lemmas for the averaged observation. These lemmas lead to the uniform convergence of the contrast and therefore to strong consistency. 
 Section~\ref{sec:numerical-example} presents a numerical illustration for a standard continuously monitored qubit model and describes the accompanying implementation. Appendix~A contains auxiliary regularity and uniform vanishing lemmas.

\section{Preliminaries}
\label{sec:preliminaries}

\subsection{Notation}
\paragraph{(N1)}
We use \(\abs{\cdot}\) for the Euclidean norm of vectors. When the
ambient Euclidean space is to be made explicit, we write
\(\norm{\cdot}_{\mathbb{R}^q}\) for the Euclidean norm on \(\mathbb{R}^q\).
We use \(\norm{\cdot}_{\mathrm{HS}}\) and \(\norm{\cdot}_{\mathrm{op}}\) for the
Hilbert--Schmidt and operator norms of matrices, respectively. We denote by \(\innerpHS{,}\) the
Hilbert--Schmidt inner product on \(M_d(\mathbb{C})\):
\begin{equation*}
    \innerpHS{A,B}
    :=
    \Tr(A^*B),
    \qquad
    A,B\in M_d(\mathbb{C}).
\end{equation*}
We also define the associated real bilinear form by
\begin{equation*}
    \bform{A,B}
    :=
    2\Real\innerpHS{A,B},
    \qquad
    A,B\in M_d(\mathbb{C}).
\end{equation*}

\paragraph{(N2)}
For any \(a,b\geq0\), the notation \(a\lesssim b\) means that
\(a\leq Cb\) for some constant \(C>0\) independent of the relevant
varying parameters.

\paragraph{(N3)}
We write
\begin{equation*}
    \Lvec^\beta
    :=
    (L_1^\beta,\dots,L_r^\beta)
    \in M_d(\mathbb{C})^r,
\end{equation*}
and, for \(\rho\in M_d(\mathbb{C})\),
\begin{equation*}
    \bform{\Lvec^\beta,\rho}
    :=
    \bigl(
        \bform{L_1^\beta,\rho},
        \dots,
        \bform{L_r^\beta,\rho}
    \bigr)^\top
    \in\mathbb{R}^r.
\end{equation*}
For \(k=1,\dots,r\), let
\begin{equation*}
    \mathcal H_k^\beta(\rho)
    :=
    L_k^\beta\rho
    +
    \rho(L_k^\beta)^*
    -
    \bform{L_k^\beta,\rho}\rho,
\end{equation*}
and write
\begin{equation*}
    \mathcal H^\beta(\rho)
    :=
    \bigl(
        \mathcal H_1^\beta(\rho),
        \dots,
        \mathcal H_r^\beta(\rho)
    \bigr).
\end{equation*}

\paragraph{(N4)}
For \(i=1,\dots,N\), let \(\rhosol{t,\theta,i}\) be the \(i\)-th independent solution of 
\begin{equation}
\label{eq:sample_path}
    d\rhosol{t,\theta,i}
    =
    \mathcal L^\theta\bigl(\rhosol{t,\theta,i}\bigr)\,dt
    +
    \mathcal H^\beta\bigl(\rhosol{t,\theta,i}\bigr)
    \cdot d\Wvec[i]_t,
\end{equation}
where
\begin{equation*}
    \Wvec[i]_t
    :=
    (W_t^{(i,1)},\dots,W_t^{(i,r)})^\top,
    \qquad
    i=1,\dots,N,
\end{equation*}
and \(\Wvec[1],\dots,\Wvec[N]\) are independent \(r\)-dimensional standard
Wiener processes on \([0,T]\).

Under the parameter \(\theta\), the corresponding measurement record is
denoted by
\begin{equation*}
    \Yvec[i]_t
    :=
    \bigl(
        Y_t^{(i,1)},\dots,Y_t^{(i,r)}
    \bigr)^\top,
    \qquad
    i=1,\dots,N,
\end{equation*}
and satisfies, componentwise,
\begin{equation*}
    dY_t^{(i,k)}
    =
    \bform{L_k^\beta,\rhosol{t,\theta,i}}\,dt
    +
    dW_t^{(i,k)},
    \qquad
    k=1,\dots,r.
\end{equation*}

Results on existence and uniqueness of the solution
\eqref{eq:sample_path} can be found in
\cite{barchielli2009quantum,barchielli1995constructing,pellegrini2010markov}.

\begin{theorem}\label{thm:existence-sme}
Let \(\rho_0\in M_d(\mathbb{C})\) be a density operator. Then, for every
\(\theta\in\Theta\) and \(i=1,\dots,N\), equation
\eqref{eq:sample_path} with initial condition
\(\rhosol{0,\theta,i}=\rho_0\) admits a unique solution.
Moreover, \(\rhosol{t,\theta,i}\) is a density operator for all
\(t\geq0\).
\end{theorem}

\paragraph{(N5)}
We define
\begin{equation*}
    \rhobar{t,\theta,N}
    :=
    \frac{1}{N}
    \sum_{i=1}^N
    \rhosol{t,\theta,i},
    \qquad
    \Wbar_t
    :=
    \frac{1}{N}
    \sum_{i=1}^N
    \Wvec[i]_t,
    \qquad
    \Ybar_t
    :=
    \frac{1}{N}
    \sum_{i=1}^N
    \Yvec[i]_t.
\end{equation*}
By averaging \eqref{eq:sample_path}, we have
\begin{equation*}
    d\rhobar{t,\theta,N}
    =
    \mathcal L^\theta\bigl(\rhobar{t,\theta,N}\bigr)\,dt
    +
    \frac{1}{N}
    \sum_{i=1}^N
    \mathcal H^\beta\bigl(\rhosol{t,\theta,i}\bigr)
    \cdot d\Wvec[i]_t.
\end{equation*}
We denote the theoretical average state by
\begin{equation*}
    \rhoave{t,\theta}
    :=
    \mathbb E\left[
        \rhosol{t,\theta,i}
    \right],
\end{equation*}
which is independent of \(i\). We further define
\begin{equation}
\label{eq:h-definition}
    h(t,\theta)
    :=
    \bform{\Lvec^\beta,\rhoave{t,\theta}}
    =
    \mathbb E\left[
        \bform{\Lvec^\beta,\rhosol{t,\theta,i}}
    \right].
\end{equation}

\subsection{Assumptions}

\begin{assumption}\label{ass:parameter-space}
Let \(\Theta^\circ\) be a bounded convex open subset of a finite-dimensional
Euclidean space, and set \(\Theta=\overline{\Theta^\circ}\). We assume that
\(\Theta\) is compact and that the true parameter \(\theta_0\) belongs to
\(\Theta^\circ\). We write \(\theta=(\alpha,\beta)\) and
\(\theta_0=(\alpha_0,\beta_0)\), where \(\alpha\) and \(\beta\) represent the
parameters entering the Hamiltonian and the measurement operators,
respectively. 
\end{assumption}

\begin{assumption}\label{distinctiveness}
If \(\theta\neq\theta_0\), then
\begin{equation}
    \bform{\Lvec^\beta,\rhoave{t,\theta}}
    \neq
    \bform{\Lvec^{\beta_0},\rhoave{t,\theta_0}}
\end{equation}
for at least one value of \(t\in[0,T]\).
\end{assumption}

\begin{assumption}\label{differentiability}
 Let \(\ell\ge0\). The coefficient matrices \(H^\alpha\) and \(L^\beta\) 
 are \(\ell\)-times continuously differentiable in \(\theta\).
 (i) $\ell=1$ (ii) $\ell=2$ (iii) $\ell=3$.
\end{assumption}

\begin{remark}
Under Assumption~\ref{differentiability},
Lemma~\ref{lem:averaged-trajectory-time-regularity} gives the
corresponding parameter and time regularity of \(\rhoave{t,\theta}\) and
\(h(t,\theta)\).
\end{remark}

Let \(p\) be the dimension of the parameter space. When a single parameter
coordinate is fixed and the argument is differentiated in that coordinate, we
write \(\partial_\theta\). Thus \(\partial_\theta\) is a componentwise notation:
it represents one of the ordinary partial derivatives with respect to a
coordinate of \(\theta\), but the coordinate is not displayed. Similarly,
\(\partial_\theta^\ell\) denotes an \(\ell\)-th order componentwise derivative.

For a scalar-valued function \(f(\theta)\), we write
\begin{equation}
    \nabla_\theta f(\theta)
    :=
    \left(
        \partial_{\theta_1}f(\theta),\ldots,
        \partial_{\theta_p}f(\theta)
    \right)^\top
\end{equation}
for the gradient column vector, namely the vector obtained by collecting all
componentwise first derivatives. Its Hessian matrix is denoted by
\(\nabla_\theta^2 f(\theta)\). In contrast to \(\partial_\theta\), the notation
\(\nabla_\theta\) is used only when the whole vector of first derivatives is
needed, and \(\nabla_\theta^2\) is used when the whole second-derivative matrix
is needed. If the differentiated quantity is vector-valued, these conventions
are applied componentwise to each scalar component.

\begin{assumption}\label{positiveDefinite}
The matrix
\begin{equation}
    I(\theta_0)
    :=
    \int_0^T
        \nabla_\theta h(t,\theta_0)
        \nabla_\theta h(t,\theta_0)^{\top}
    \,dt
\end{equation}
is positive definite.
\end{assumption}

\section{Main results}
\label{sec:main-results}

\subsection{Continuously observed case}
 For each \(i=1,\ldots,N\), suppose that one realization
\(\{\Yvec[i]_t:0\leq t\leq T\}\) is available. We estimate
\(\theta\) from these \(N\) realizations. For a fixed
\(\theta\in\Theta\), Girsanov's theorem gives the Radon--Nikodym
derivative of the distribution of \(\Yvec[i]\) with respect to the
probability measure induced by \(\Wvec[i]\) as
\begin{equation}
    \exp\qty{
        \int_0^T
        \bform{\Lvec^\beta,\rhosol{t,\theta,i}}
        \cdot d\Yvec[i]_t
        -
        \frac{1}{2}
        \int_0^T
        \norm{\bform{\Lvec^\beta,\rhosol{t,\theta,i}}}_{\mathbb{R}^r}^2\,dt
    }.
\end{equation}
Since the \(N\) realizations are independent, the normalized
log-likelihood is
\begin{equation}
    \frac{1}{N}
    \sum_{i=1}^N
    \left\{
        \int_0^T
        \bform{\Lvec^\beta,\rhosol{t,\theta,i}}
        \cdot d\Yvec[i]_t
        -
        \frac{1}{2}
        \int_0^T
        \norm{\bform{\Lvec^\beta,\rhosol{t,\theta,i}}}_{\mathbb{R}^r}^2\,dt
    \right\}.
\end{equation}
However, the state trajectory
\(\rhosol{t,\theta,i}\) appearing in the likelihood is not directly
observed and has to be computed from each measurement record for each
candidate value of \(\theta\). We instead replace
\(\rhosol{t,\theta,i}\) by the averaged trajectory
\(\rhoave{t,\theta}\) and consider the contrast function
\begin{equation}\label{eq:contrast}
    \Phi_N(\theta)
    :=
    \int_0^T h(t,\theta)\cdot d\Ybar_t
    -
    \frac{1}{2}\int_0^T \norm{h(t,\theta)}_{\mathbb{R}^r}^2\,dt.
\end{equation}
We define the maximum likelihood type estimator by
\begin{equation}\label{eq:mce}
    \hat{\theta}_N
    =
    \argmax_{\theta\in\Theta}\Phi_N(\theta).
\end{equation}

The strong consistency of our estimator is given as follows.
\begin{theorem}\label{consistency}
Suppose that Assumptions~\ref{ass:parameter-space}, \ref{distinctiveness},and \ref{differentiability}(i) hold. Then
\begin{equation}
    \hat{\theta}_N\convas\theta_0,
\end{equation}
as \(N\to\infty\).
\end{theorem}

For each \(\theta\in\Theta\) and $i=1,\dots,N$, define
\begin{equation}\label{eq:eta-i-theta-definition}
    \eta_i(\theta)
    :=
    \int_0^T
        \nabla_\theta h(t,\theta)
        \left\{d\Yvec[i]_t-h(t,\theta)\,dt\right\},
    \qquad 
    \Sigma
:=
E[\eta_i(\theta_0)\eta_i(\theta_0)^\top]
\end{equation}
The next theorem gives the asymptotic normality of our estimator $\hat{\theta}_N$.
\begin{theorem}\label{asymptotic normality}
Suppose that Assumptions~\ref{ass:parameter-space}, \ref{distinctiveness}, \ref{positiveDefinite}, and \ref{differentiability}(iii) hold. Then
\begin{equation}
    \sqrt{N}\left(\hat{\theta}_N-\theta_0\right)
    \xrightarrow{d}
    \mathcal{N}\left(0,I(\theta_0)^{-1}\Sigma I(\theta_0)^{-1}\right),
\end{equation}
as \(N\to\infty\).
\end{theorem}

Since \(\Sigma\) is defined in terms of the single-trajectory quantity
\(\eta_i(\theta_0)\), we estimate it by the sample covariance of its
plug-in counterparts. To estimate \(\Sigma\), define
\begin{equation}
    \hat\eta_i
    :=
    \eta_i(\hat\theta_N),
    \qquad
    \bar\eta_N
    :=
    \frac{1}{N}\sum_{i=1}^N \hat\eta_i .
\end{equation}
We then define
\begin{equation}\label{eq:Sigma-hat-definition}
    \hat\Sigma_N
    :=
    \frac{1}{N}
    \sum_{i=1}^N
    (\hat\eta_i-\bar\eta_N)
    (\hat\eta_i-\bar\eta_N)^\top .
\end{equation}

We also define the plug-in information matrix by
\begin{equation}\label{eq:I-hat-definition}
    \hat I_N
    :=
    \int_0^T
        \nabla_\theta h(t,\hat{\theta}_N)
        \nabla_\theta h(t,\hat{\theta}_N)^\top
    \,dt.
\end{equation}
Then a covariance estimator for the asymptotic covariance matrix of
\(\sqrt{N}(\hat{\theta}_N-\theta_0)\) is the sandwich matrix
\begin{equation}\label{eq:V-hat-definition}
    \hat V_N
    :=
    \hat I_N^{-1}\hat\Sigma_N\hat I_N^{-1}.
\end{equation}

\begin{theorem}\label{thm:covariance-estimator-consistency}
Suppose that Assumptions~\ref{ass:parameter-space}, \ref{distinctiveness},
\ref{positiveDefinite}, and \ref{differentiability}(ii)
hold. Then
\begin{equation}
    \hat\Sigma_N\convas\Sigma
\end{equation}
and
\begin{equation}
    \hat V_N\convas I(\theta_0)^{-1}\Sigma I(\theta_0)^{-1}.
\end{equation}
\end{theorem}

\subsection{Discretely observed case}
\label{subsec:main-results-discrete}
\label{subsec:discrete-observation-contrast}
In the previous subsection, we constructed an estimator from \(N\)
realizations on \([0,T]\) and established its asymptotic properties.
In practice, however, although the measurement is modeled in continuous
time, measurement records are acquired with finite time resolution.
We therefore consider estimation based on discretely sampled data. Let
\begin{equation}
0=t_0<t_1<\cdots<t_n=T,
\qquad
t_j=j\Delta_n,
\qquad
\Delta_n=\frac{T}{n}.
\end{equation}
Suppose that the available data are
\(\{Y_{t_j}^{(i)}:j=0,\ldots,n\}\), \(i=1,\ldots,N\). For the averaged
data, put
\begin{equation}
\Delta_j\Ybar
:=
\Ybar_{t_j}-\Ybar_{t_{j-1}},
\qquad
j=1,\ldots,n.
\end{equation}
To construct a contrast function from these discrete data, we replace
\(h(t,\theta)\) on each interval \((t_{j-1},t_j]\) by its left-endpoint
value and define
\begin{equation}\label{eq:discrete-step-h}
h_n(t,\theta)
:=
\sum_{j=1}^{n}
h(t_{j-1},\theta)\mathbf{1}_{(t_{j-1},t_j]}(t).
\end{equation}
with \(h_n(0,\theta):=h(0,\theta)\). Accordingly, we discretize the contrast function in the previous
subsection as
\begin{align}
\Phi_{N,n}(\theta)
&:=
\sum_{j=1}^{n}
h(t_{j-1},\theta)\cdot\Delta_j\Ybar
-
\frac{1}{2}
\sum_{j=1}^{n}
\norm{h(t_{j-1},\theta)}_{\mathbb{R}^r}^2\Delta_n
\notag\\
&=
\int_0^T h_n(t,\theta)\cdot d\Ybar_t
-
\frac{1}{2}\int_0^T \norm{h_n(t,\theta)}_{\mathbb{R}^r}^2\,dt.
\label{eq:discrete-contrast}
\end{align}
We define the estimator based on the discrete data by
\begin{equation}\label{eq:discrete-mce}
\hat\theta_{N,n}
:=
\argmax_{\theta\in\Theta}\Phi_{N,n}(\theta).
\end{equation}

\begin{theorem}\label{thm:discrete-consistency}
Suppose that Assumptions~\ref{ass:parameter-space},\ref{distinctiveness}
and \ref{differentiability}(i) hold. Then, as
\(N\to\infty\) and \(n\to\infty\),
\begin{equation}
    \hat\theta_{N,n}\convas\theta_0.
\end{equation}
\end{theorem}

Differentiating \eqref{eq:discrete-contrast}, we have
\begin{equation}\label{eq:discrete-score}
\begin{aligned}
\nabla_\theta\Phi_{N,n}(\theta)
    &=
    \int_0^T
        \nabla_\theta h_n(t,\theta)
        \left\{d\Ybar_t-h_n(t,\theta)\,dt\right\}
        \\
        &=\sum_{j=1}^{n}
        \nabla_\theta h(t_{j-1},\theta)
        \left\{
            \Delta_j\Ybar-h(t_{j-1},\theta)\Delta_n
        \right\}.
\end{aligned}
\end{equation}
The corresponding discrete Hessian is
\begin{align}
    \nabla_\theta^2\Phi_{N,n}(\theta)
    &=
    \int_0^T
        \nabla_\theta^2 h_n(t,\theta)
        \cdot
        \left\{d\Ybar_t-h_n(t,\theta)\,dt\right\} \notag\\
    &\quad
    -
    \int_0^T
        \nabla_\theta h_n(t,\theta)
        \nabla_\theta h_n(t,\theta)^\top
        \,dt.
\label{eq:discrete-hessian}
\end{align}

For covariance estimation, define 
\begin{align}
    \eta_{i,n}(\theta)
    &:=
    \int_0^T
        \nabla_\theta h_n(t,\theta)
        \left\{d\Yvec[i]_t-h_n(t,\theta)\,dt\right\} \notag\\
    &=
    \sum_{j=1}^{n}
        \nabla_\theta h(t_{j-1},\theta)
        \left\{
            \Delta_j\Yvec[i]-h(t_{j-1},\theta)\Delta_n
        \right\},
    \qquad
    i=1,\ldots,N,
\label{eq:discrete-eta-i-theta}
\end{align}
where
\(
    \Delta_j\Yvec[i]
    :=
    \Yvec[i]_{t_j}-\Yvec[i]_{t_{j-1}}
\).
Then
\begin{equation}\label{eq:discrete-score-average}
    \nabla_\theta\Phi_{N,n}(\theta)
    =
    \frac{1}{N}\sum_{i=1}^N \eta_{i,n}(\theta).
\end{equation}
Evaluating the score at the estimator $\hat\theta_{N,n}$, we define
\begin{equation}
    \hat\eta_{i,n}:=\eta_{i,n}(\hat\theta_{N,n}),
    \qquad
    \bar\eta_{N,n}:=\frac{1}{N}\sum_{i=1}^{N}\hat\eta_{i,n}.
\end{equation}
The sample covariance estimator and the plug-in information matrix are given by
\begin{equation}\label{eq:discrete-Sigma-hat}
    \hat\Sigma_{N,n}
    :=
    \frac{1}{N}\sum_{i=1}^{N}
    (\hat\eta_{i,n}-\bar\eta_{N,n})
    (\hat\eta_{i,n}-\bar\eta_{N,n})^\top, \quad\hat I_{N,n}
    :=
    \int_0^T
        \nabla_\theta h_n(t,\hat\theta_{N,n})
        \nabla_\theta h_n(t,\hat\theta_{N,n})^\top
        \,dt.
\end{equation}
Let $\hat V_{N,n}$ be the corresponding sandwich estimator defined by
\begin{equation}\label{eq:discrete-V-hat}
    \hat V_{N,n}
    =
    \hat I_{N,n}^{-1}\hat\Sigma_{N,n}\hat I_{N,n}^{-1},
\end{equation}
provided that \(\hat I_{N,n}\) is invertible. The asymptotic normality of $\hat\theta_{N,n}$ is given as follows.

\begin{theorem}\label{thm:discrete-asymptotic-normality}
Suppose that Assumptions~\ref{ass:parameter-space}, \ref{distinctiveness},
\ref{positiveDefinite}, and \ref{differentiability}(iii)
hold. If \(N\to\infty\), \(n\to\infty\), and
\begin{equation}\label{eq:mesh-rate-sqrtN}
    \sqrt{N}\,\Delta_n\to0,
\end{equation}
then
\begin{equation}
    \sqrt{N}(\hat\theta_{N,n}-\theta_0)
    \xrightarrow{d}
    \mathcal{N}\left(0,I(\theta_0)^{-1}\Sigma I(\theta_0)^{-1}\right).
\end{equation}
\end{theorem}

The discrete covariance estimator is consistent in the following sense.

\begin{theorem}\label{thm:discrete-covariance-estimator-consistency}
Suppose that Assumptions~\ref{ass:parameter-space}, \ref{distinctiveness},
\ref{positiveDefinite}, and \ref{differentiability}(ii)
hold. Then, as \(N\to\infty\) and
\(n\to\infty\),
\begin{equation}
    \hat\Sigma_{N,n}\convas\Sigma,
    \qquad
    \hat V_{N,n}\convas I(\theta_0)^{-1}\Sigma I(\theta_0)^{-1}.
\end{equation}
\end{theorem}

\section{Proof of the main results}
\label{sec:proof-main-results}

\begingroup
\subsection{Continuously observed case}
\label{subsec:proof-continuously-observed-case}

By \zcref{lem:differentiation-stochastic-integral}, the contrast function \eqref{eq:contrast} is twice differentiable with respect to $\theta$ and we have
\begin{align}
\nabla_\theta\Phi_N(\theta)
&=
\int_0^T
\nabla_\theta h(t,\theta)
\left\{
d\Ybar_t-h(t,\theta)\,dt
\right\},
\label{eq:first-derivative-h}\\
\nabla_\theta^2\Phi_N(\theta)
&=
\int_0^T
\nabla_\theta^2 h(t,\theta)
\cdot
\left\{
d\Ybar_t-h(t,\theta)\,dt
\right\}
-
\int_0^T
\nabla_\theta h(t,\theta)
\nabla_\theta h(t,\theta)^\top
\,dt.
\label{eq:second-derivative-h}
\end{align}

The componentwise first derivative of \(h\) is given by
\begin{equation}\label{eq:h-beta}
    \partial_\theta h(t,\theta)
    =
    \bform{\partial_\theta\Lvec^\beta,\rhoave{t,\theta}}
    +
    \bform{\Lvec^\beta,\partial_\theta\rhoave{t,\theta}}.
\end{equation}

\subsubsection{Proof of \zcref{consistency}}

We now collect the elementary convergence results that will be used in the
proofs below.  We will use the
convention \(\partial_\theta^0h=h\).
\begin{lemma}\label{lem:basic-uniform-convergences}
Suppose that \zcref{ass:parameter-space} holds, and let
\(\ell\in\{0,1,2\}\). 
\begin{itemize}
    \item[(i)] If \zcref{differentiability} holds with order \(\ell\), then
\begin{equation}\label{eq:empirical-drift-convergence}
    \int_0^T
        \partial_\theta^\ell h(t,\theta)\cdot
        \bform{\Lvec^{\beta_0},\rhobar{t,\theta_0,N}}
    \,dt
    \convas
    \int_0^T
        \partial_\theta^\ell h(t,\theta)\cdot
        \bform{\Lvec^{\beta_0},\rhoave{t,\theta_0}}
    \,dt.
\end{equation}
as \(N\to\infty\), uniformly in \(\theta\in\Theta\).
\item[(ii)]If \zcref{differentiability} holds with order \(\ell+1\), then
\begin{align}
    \int_0^T
        \partial_\theta^\ell h(t,\theta)\cdot d\Wbar_t
    &\convas 0,
    \label{eq:stoch-int-vanishes-h-derivatives}\\
    \int_0^T
        \partial_\theta^\ell h(t,\theta)\cdot d\Ybar_t
    &\convas
    \int_0^T
        \partial_\theta^\ell h(t,\theta)\cdot h(t,\theta_0)
    \,dt.
    \label{eq:ybar-integral-convergence}
\end{align}
as \(N\to\infty\), uniformly in \(\theta\in\Theta\).
\end{itemize}
\end{lemma}

\begin{proof}
\begingroup
Put
\begin{align}
    X_i(t)
    &:=
    \bform{\Lvec^{\beta_0},\rhosol{t,\theta_0,i}}
    -
    \bform{\Lvec^{\beta_0},\rhoave{t,\theta_0}},
    \\
    \overline{X}^N(t)
    &:=
    \frac{1}{N}\sum_{i=1}^N X_i(t)
    =
    \bform{\Lvec^{\beta_0},\rhobar{t,\theta_0,N}}
    -
    \bform{\Lvec^{\beta_0},\rhoave{t,\theta_0}},
    \qquad 0\leq t\leq T .
\end{align}
Then \(X_1,X_2,\dots\) are independent and identically distributed (i.i.d.)
\(C([0,T];\mathbb R^r)\)-valued random variables. Since both
\(\rhosol{t,\theta_0,1}\) and \(\rhoave{t,\theta_0}\) are density operators,
\begin{equation}
    \norm{X_1}_{C([0,T];\mathbb{R}^r)}
    \leq
    4\left(
        \sum_{k=1}^r
        \norm{L_k^{\beta_0}}_{\mathrm{op}}^2
    \right)^{1/2}
    \qquad\text{a.s.},
\end{equation}
and, in particular,
\[
    \mathbb{E}\norm{X_1}_{C([0,T];\mathbb{R}^r)}<\infty .
\]
Moreover, \(\mathbb{E}X_1(t)=0\) for every \(t\in[0,T]\).
Since \(C([0,T];\mathbb{R}^r)\) is separable, \(X_1\) is strongly
measurable, and hence the above integrability bound implies that \(X_1\)
is Bochner integrable with mean zero.
Therefore, the strong law of large numbers in a separable Banach space
\cite[Corollary~7.10]{ledoux1991probability} yields
\begin{equation}
    \norm{\overline{X}^N}_{C([0,T];\mathbb{R}^r)}
    \convas 0 .
\end{equation}
Consequently, for each \(\ell=0,1,2\),
\begin{equation}
\sup_{\theta\in\Theta}
\left|
    \int_0^T \partial_\theta^\ell h(t,\theta)\cdot \overline{X}^N(t)\,dt
\right|
\leq
T\norm{\partial_\theta^\ell h}_{C([0,T]\times\Theta;\mathbb{R}^r)}
\norm{\overline{X}^N}_{C([0,T];\mathbb{R}^r)}
\convas 0,
\end{equation}
which proves (i).
\endgroup
Assertion (ii) follows by combining \zcref{lem:StochIntVanishesGeneral} with (i).
\end{proof}

\begin{proof}[Proof of \zcref{consistency}]
By \zcref{lem:basic-uniform-convergences},
\begin{equation}
    \int_0^T h(t,\theta)\cdot d\Ybar_t
    \convas
    \int_0^T h(t,\theta)\cdot h(t,\theta_0)\,dt
\end{equation}
uniformly in \(\theta\in\Theta\). Therefore
\begin{equation}\label{eq:contrast-limit}
    \Phi_N(\theta)
    \convas
    \Phi(\theta)
    :=
    -\frac{1}{2}
    \int_0^T
        \norm{h(t,\theta)-h(t,\theta_0)}_{\mathbb{R}^r}^2
    \,dt
    +
    \frac{1}{2}
    \int_0^T
        \norm{h(t,\theta_0)}_{\mathbb{R}^r}^2
    \,dt,
\end{equation}
uniformly in \(\theta\in\Theta\).  
\begingroup
The function \(\Phi\) is continuous on the compact set \(\Theta\). By
\zcref{distinctiveness} and the continuity of
\(t\mapsto h(t,\theta)-h(t,\theta_0)\), the integral in
\eqref{eq:contrast-limit} is strictly positive for every
\(\theta\neq\theta_0\). Hence \(\Phi\) has the unique maximizer
\(\theta_0\). The argmax theorem therefore implies
\(\hat{\theta}_N\convas\theta_0\).
\endgroup
\end{proof}

\subsubsection{Proof of \zcref{asymptotic normality}}

\begin{proposition}\label{prop:ScoreCLT}
Suppose that \zcref{ass:parameter-space} and \zcref{differentiability}(i) hold. Then
\begin{equation}
    \sqrt{N}\,\nabla_\theta\Phi_N(\theta_0)
    \xrightarrow{d}
    \mathcal{N}\left(0,\Sigma\right).
\end{equation}
\end{proposition}

\begin{proof}
Note that
\(\eta_1(\theta_0),\dots,\eta_N(\theta_0)\) are i.i.d., with
\begin{equation}
    \mathbb{E}[\eta_i(\theta_0)]=0,
    \qquad
    \mathbb{E}\left[
        \eta_i(\theta_0)\eta_i(\theta_0)^\top
    \right]
    =
    \Sigma.
\end{equation}
Moreover, the boundedness of \(\nabla_\theta h(t,\theta_0)\) and
\(\bform{\Lvec^{\beta_0},\rhosol{t,\theta_0,i}}\), together with the
Burkholder--Davis--Gundy inequality, implies
\begin{equation}
    \mathbb{E}\left[\norm{\eta_i(\theta_0)}_{\mathbb{R}^p}^2\right]<\infty.
\end{equation}
Hence, by the multivariate central limit theorem,
\begin{equation}
    \sqrt{N}\,\nabla_\theta\Phi_N(\theta_0)
    =
    \frac{1}{\sqrt{N}}
    \sum_{i=1}^N\eta_i(\theta_0)
    \xrightarrow{d}
    \mathcal{N}(0,\Sigma).
\end{equation}
\end{proof}
\begin{proof}[Proof of \zcref{asymptotic normality}]
Since \(\theta_0\in\Theta^\circ\) and
\(\hat{\theta}_N\convas\theta_0\), we have
\(\hat{\theta}_N\in\Theta^\circ\) eventually almost surely. Hence,
\(\nabla_\theta\Phi_N(\hat{\theta}_N)=0\) eventually almost surely.
By Taylor's theorem,
\begin{equation}
    0
    =
    \sqrt{N}\,\nabla_\theta\Phi_N(\hat{\theta}_N)
    =
    \sqrt{N}\,\nabla_\theta\Phi_N(\theta_0)
    +
    \left\{
        \int_0^1
        \nabla_\theta^2\Phi_N
        \bigl(
            \theta_0+s(\hat{\theta}_N-\theta_0)
        \bigr)\,ds
    \right\}
    \sqrt{N}(\hat{\theta}_N-\theta_0).
\end{equation}
By \zcref{lem:basic-uniform-convergences} and
\zcref{consistency},
\begin{equation}
\label{HessianConvergence}
\int_0^1
\nabla_\theta^2\Phi_N
\bigl(
\theta_0+s(\hat\theta_N-\theta_0)
\bigr)\,ds
\convas
-I(\theta_0).
\end{equation}
\begingroup
Indeed, by \zcref{lem:basic-uniform-convergences}, uniformly on
\(\Theta\),
\[
\nabla_\theta^2\Phi_N(\theta)
\convas
\int_0^T
\nabla_\theta^2h(t,\theta)\cdot\{h(t,\theta_0)-h(t,\theta)\}\,dt
-
\int_0^T
\nabla_\theta h(t,\theta)\nabla_\theta h(t,\theta)^\top\,dt,
\]
and the limit is continuous in \(\theta\) and equals \(-I(\theta_0)\) at
\(\theta=\theta_0\), so \eqref{HessianConvergence} follows from
\zcref{consistency}.
\endgroup
Since \(I(\theta_0)\) is positive definite by
\zcref{positiveDefinite}, \zcref{prop:ScoreCLT}
and Slutsky's theorem yield
\begin{equation}
\sqrt{N}(\hat{\theta}_N-\theta_0)
\xrightarrow{d}
\mathcal{N}\left(
0,
I(\theta_0)^{-1}\Sigma I(\theta_0)^{-1}
\right).
\end{equation}
\end{proof}

\subsubsection{Proof of \zcref{thm:covariance-estimator-consistency}}

\begin{lemma}\label{lem:SecondMomentUniformLLN}
Suppose that \zcref{ass:parameter-space} and \zcref{differentiability}(ii) hold. Put
\begin{equation}
    S_N(\theta)
    :=
    \frac{1}{N}\sum_{i=1}^N
    \eta_i(\theta)\eta_i(\theta)^\top,
    \qquad \theta\in\Theta,
\end{equation}
and
\begin{equation}
    S(\theta)
    :=
    \mathbb{E}\left[\eta_1(\theta)\eta_1(\theta)^\top\right].
\end{equation}
Then \(S\) is continuous on \(\Theta\), and
\begin{equation}\label{eq:uniform-lln-SN-lemma}
    \sup_{\theta\in\Theta}
    \left\|S_N(\theta)-S(\theta)\right\|_{\mathrm{HS}}
    \convas 0 .
\end{equation}
\end{lemma}

\begin{proof}
For $i=1,\dots,N$, we write
\begin{equation}
    \eta_i(\theta)
    =
    A_i(\theta)+B_i(\theta),
\end{equation}
where
\begin{align}
    A_i(\theta)
    &:=
    \int_0^T
        \nabla_\theta h(t,\theta)
        \left\{
            \bform{\Lvec^{\beta_0},\rhosol{t,\theta_0,i}}
            -h(t,\theta)
        \right\}
    \,dt, \\
    B_i(\theta)
    &:=
    \int_0^T
        \nabla_\theta h(t,\theta)
    \,d\Wvec[i]_t .
\end{align}
The map \(\theta\mapsto A_i(\theta)\) is continuous and uniformly bounded,
because \(\nabla_\theta h\) and \(h\) are continuous and bounded on
\([0,T]\times\Theta\), and because \(\rhosol{t,\theta_0,i}\) is a density
matrix. Next, we justify the continuity of \(B_i\). Choose \(\gamma>\max\{p,4\}\). 
\begingroup
Applying \zcref{lem:differentiation-stochastic-integral} with \(\ell=1\)
to \(f=\nabla_\theta h\), we obtain
\endgroup
\begin{equation}
    \partial_\theta B_i(\theta)
    =
    \int_0^T
        \nabla_\theta^2 h(t,\theta)
    \,d\Wvec[i]_t .
\end{equation}
 Hence the boundedness of the first and second
\(\theta\)-derivatives of \(h\), together with the Burkholder--Davis--Gundy
inequality, yields
\begingroup
\begin{equation}
    \mathbb{E}\left[
        \|B_i\|_{W^{1,\gamma}(\Theta^\circ;\mathbb{R}^p)}^\gamma
    \right]<\infty .
\end{equation}
\endgroup
Since \(\gamma>p\) and \(\Theta\) is the closure of a bounded convex domain,
the Sobolev--Morrey embedding
\(W^{1,\gamma}(\Theta^\circ;\mathbb{R}^p)\hookrightarrow
C(\Theta;\mathbb{R}^p)\) implies that \(B_i\) admits a continuous modification and
\begin{equation}
    \mathbb{E}\left[
        \sup_{\theta\in\Theta}\norm{B_i(\theta)}_{\mathbb{R}^p}^\gamma
    \right]<\infty .
\end{equation}
Therefore, \(\eta_i(\theta)\eta_i(\theta)^\top\) is an integrable random element of
\(C(\Theta;\mathbb{R}^{p\times p})\), and hence
\(S(\theta)\) is continuous.
\begingroup
Since \(C(\Theta;\mathbb{R}^{p\times p})\) is separable and these random
elements are i.i.d., the strong law of large numbers in a separable Banach
space \cite[Corollary~7.10]{ledoux1991probability} yields
\begin{equation}
    \sup_{\theta\in\Theta}
    \left\|S_N(\theta)-S(\theta)\right\|_{\mathrm{HS}}
    \convas 0 .
\end{equation}
\endgroup
\end{proof}

\begin{proof}[Proof of \zcref{thm:covariance-estimator-consistency}]
For \(\theta\in\Theta\), put
\begin{equation}
    m_N(\theta)
    :=
    \frac{1}{N}\sum_{i=1}^N \eta_i(\theta).
\end{equation}
Then \(\hat\Sigma_N\) can be written as
\begin{equation}
    \hat\Sigma_N
    =
    S_N(\hat{\theta}_N)-m_N(\hat{\theta}_N)m_N(\hat{\theta}_N)^\top.
\end{equation}
By the definition of \(\eta_i(\theta)\),
\begin{equation}
    m_N(\theta)
    =
    \int_0^T \nabla_\theta h(t,\theta)\, d\Ybar_t
    -
    \int_0^T \nabla_\theta h(t,\theta) h(t,\theta)\,dt .
\end{equation}
Hence \zcref{lem:basic-uniform-convergences} with \(\ell=1\) gives
\begin{equation}\label{eq:mN-uniform-convergence-covariance-proof}
    \sup_{\theta\in\Theta}
    \norm{
        m_N(\theta)-m(\theta)
    }_{\mathbb{R}^p}
    \convas 0,
    \qquad
    m(\theta)
    :=
    \mathbb{E}\left[\eta_1(\theta)\right].
\end{equation}
\begingroup
The map \(m\) is continuous on \(\Theta\), and
\(m(\theta_0)=\mathbb{E}[\eta_1(\theta_0)]=0\). Hence, by
\zcref{consistency} and the uniform convergence displayed above,
\[
    m_N(\hat\theta_N)\convas0.
\]
\endgroup
Moreover, by \zcref{lem:SecondMomentUniformLLN},
\begin{equation}
    \sup_{\theta\in\Theta}
    \left\|S_N(\theta)-S(\theta)\right\|_{\mathrm{HS}}
    \convas 0,
\end{equation}
and we have
\begin{align}
    &\left\|
        S_N(\hat\theta_N)-S(\theta_0)
    \right\|_{\mathrm{HS}} \notag\\
    &\qquad\leq
    \sup_{\theta\in\Theta}
    \left\|S_N(\theta)-S(\theta)\right\|_{\mathrm{HS}}
    +
    \left\|S(\hat\theta_N)-S(\theta_0)\right\|_{\mathrm{HS}}
    \convas 0 ,
\end{align}
\begingroup
Since \(S(\theta_0)=\Sigma\), combining the two preceding convergences with
\begin{equation}
    \hat\Sigma_N
    =S_N(\hat\theta_N)
    -m_N(\hat\theta_N)m_N(\hat\theta_N)^\top
\end{equation}
gives
\begin{equation}
    \hat\Sigma_N
    \convas
    \Sigma .
\end{equation}
\endgroup

By \zcref{consistency}, 
\begin{equation}
    \hat I_N
    \convas
    \int_0^T
        \nabla_\theta h(t,\theta_0)
        \nabla_\theta h(t,\theta_0)^\top
    \,dt
    =
    I(\theta_0).
\end{equation}
\begingroup
Since \(I(\theta_0)\) is positive definite, matrix inversion is continuous
at \(I(\theta_0)\). Hence,
\begin{equation}
    \hat I_N^{-1}\convas I(\theta_0)^{-1}.
\end{equation}
\endgroup
Combining this with \(\hat\Sigma_N\convas\Sigma\) yields
\begin{equation}
    \hat V_N
    =
    \hat I_N^{-1}\hat\Sigma_N\hat I_N^{-1}
    \convas
    I(\theta_0)^{-1}\Sigma I(\theta_0)^{-1}.
\end{equation}
\end{proof}

\subsection{Discretely observed case}
\label{subsec:proof-discretely-observed-case}

\subsubsection{Proof of \zcref{thm:discrete-consistency}}
\label{subsubsec:proof-discrete-consistency}

If \zcref{differentiability} holds with order \(\ell\ge0\), then
for \(t\in(t_{j-1},t_j]\),
\begin{equation}\label{eq:discrete-step-derivative-error}
    \partial_\theta^\ell h_n(t,\theta)
    =
    \partial_\theta^\ell h(t_{j-1},\theta),
\end{equation}
and by \zcref{lem:averaged-trajectory-time-regularity}, we have
\begin{equation}\label{eq:discrete-local-score-step-rate}
    \sup_{(t,\theta)\in[0,T]\times\Theta}
    \norm{
        \partial_\theta^\ell h_n(t,\theta)
        -
        \partial_\theta^\ell h(t,\theta)
    }_{\mathbb{R}^r}
    \le C\Delta_n .
\end{equation}

\begin{lemma}
\label{lem:discrete_vanish}
Suppose that \zcref{ass:parameter-space} holds.
\begin{itemize}
    \item[(i)] For each \(\ell=0,1,2\), suppose that \zcref{differentiability} holds with \(\ell\). Then
    \begin{equation}
    \sup_{N\ge1}\sup_{\theta\in\Theta}
    \abs{
    \int_0^T
        \{\partial_\theta^\ell h_n(t,\theta)
        -\partial_\theta^\ell h(t,\theta)\}
        \cdot
        \bform{\Lvec^{\beta_0},\rhobar{t,\theta_0,N}}
    \,dt
    }
    \to0
    \end{equation}
    as \(n\to\infty\).

    \item[(ii)] For each \(\ell=0,1\), suppose that \zcref{differentiability} holds with \(\ell+1\). Then
    \begin{equation}
    \begin{aligned}
        &\lim_{K\to\infty}
        \sup_{\substack{N\ge K\\ n\ge K}}
        \sup_{\theta\in\Theta}
        \abs{
            \int_0^T
            \{\partial_\theta^\ell h_n(t,\theta)-\partial_\theta^\ell h(t,\theta)\}
            \cdot d\Wbar_t
        }
        =0
        \quad\text{a.s.},
        \\
        &\lim_{K\to\infty}
        \sup_{\substack{N\ge K\\ n\ge K}}
        \sup_{\theta\in\Theta}
        \abs{
            \int_0^T
            \{\partial_\theta^\ell h_n(t,\theta)-\partial_\theta^\ell h(t,\theta)\}
            \cdot d\Ybar_t
        }
        =0
        \quad\text{a.s.}
    \end{aligned}
    \end{equation}
\end{itemize}
\end{lemma}

\begin{proof}
Since \(\rhobar{t,\theta_0,N}\) is a density operator,
\begin{equation}
\norm{\bform{\Lvec^{\beta_0},\rhobar{t,\theta_0,N}}}_{\mathbb{R}^r}
\le
2\left(
\sum_{k=1}^r
\norm{L_k^{\beta_0}}_{\mathrm{op}}^2
\right)^{1/2}.
\end{equation}
Therefore, by \eqref{eq:discrete-local-score-step-rate},
\begin{align}
&\sup_{N\ge1}\sup_{\theta\in\Theta}
\abs{
    \int_0^T
        \{\partial_\theta^\ell h_n(t,\theta)
        -\partial_\theta^\ell h(t,\theta)\}
        \cdot
        \bform{\Lvec^{\beta_0},\rhobar{t,\theta_0,N}}
    \,dt
}
\lesssim\Delta_n.
\end{align}
The right-hand side converges to zero as \(n\to\infty\), which proves (i).
For (ii), by \(\Delta_n=T/n\) and \eqref{eq:discrete-local-score-step-rate}, we can choose the integer \(m\) in
\zcref{lem:StochIntVanishesCombined}(ii) so that
\(\sum_{n=1}^\infty \Delta_n^{2m}<\infty\). Thus
\begin{equation}
    \lim_{K\to\infty}
    \sup_{\substack{N\ge K\\ n\ge K}}
    \sup_{\theta\in\Theta}
    \abs{
        \int_0^T
        \{\partial_\theta^\ell h_n(t,\theta)
        -\partial_\theta^\ell h(t,\theta)\}
        \cdot d\Wbar_t
    }
    =0
    \quad\text{a.s.}
\end{equation}
The second convergence in (ii) follows from the first one and (i).
\end{proof}

\begin{proof}[Proof of \zcref{thm:discrete-consistency}]
By the definitions of \(\Phi_{N,n}\) and \(\Phi_N\),
\begin{equation}\label{eq:discrete-consistency-contrast-difference}
    \Phi_{N,n}(\theta)-\Phi_N(\theta)
    =
    \int_0^T \{h_n(t,\theta)-h(t,\theta)\}\cdot d\Ybar_t
    -
    \frac12
    \int_0^T
        \{\norm{h_n(t,\theta)}_{\mathbb{R}^r}^2
        -\norm{h(t,\theta)}_{\mathbb{R}^r}^2\}\,dt .
\end{equation}
The first term on the right-hand side converges to zero uniformly in
\(\theta\in\Theta\), almost surely, as \(N\to\infty\) and \(n\to\infty\), by
\zcref{lem:discrete_vanish}:
\begin{equation}\label{eq:discrete-consistency-ybar-term-vanishes}
    \sup_{\theta\in\Theta}
    \abs{
        \int_0^T \{h_n(t,\theta)-h(t,\theta)\}\cdot d\Ybar_t
    }
    \convas0 .
\end{equation}
For the second term, the uniform continuity and boundedness of \(h\) on
\([0,T]\times\Theta\) imply
\begin{equation}\label{eq:discrete-consistency-square-term}
    \sup_{\theta\in\Theta}
    \abs{
    \int_0^T
        \{\norm{h_n(t,\theta)}_{\mathbb{R}^r}^2
        -\norm{h(t,\theta)}_{\mathbb{R}^r}^2\}\,dt
    }
    \to0 .
\end{equation}
Combining this with the proof of \zcref{consistency}, we obtain
\begin{equation}\label{eq:discrete-consistency-uniform-convergence}
    \sup_{\theta\in\Theta}
    \abs{\Phi_{N,n}(\theta)-\Phi(\theta)}\convas0.
\end{equation}
As in the proof of \zcref{consistency}, the deterministic contrast
\(\Phi\) has the unique maximizer \(\theta_0\) by
\zcref{distinctiveness}. Since \(\Theta\) is compact, the same argmax
argument gives
\begin{equation}\label{eq:discrete-consistency-final}
    \hat\theta_{N,n}\convas\theta_0.
\end{equation}
\end{proof}

\subsubsection{Asymptotic normality}
\label{subsubsec:proof-discrete-asymptotic-normality}

\begin{lemma}\label{lem:discrete-local-hessian-comparison}
    Suppose that \zcref{ass:parameter-space} holds.
    \begin{itemize}
        \item [(i)] If \zcref{differentiability}(ii) holds,
\(n=n(N)\to\infty\) and \(\sqrt{N}\Delta_n\to0\), then, for every neighbourhood \(U\Subset\Theta^\circ\) of \(\theta_0\),
\begin{equation}
    \sqrt{N}
    \sup_{\theta\in U}
    \norm{
        \nabla_\theta\Phi_{N,n}(\theta)
        -\nabla_\theta\Phi_N(\theta)
    }_{\mathbb{R}^p}
    \convp0 .
\end{equation}
\item[(ii)]
If \zcref{differentiability}(iii) holds and
\(n=n(N)\to\infty\), then, for every neighbourhood
\(U\Subset\Theta^\circ\) of \(\theta_0\),
\begin{equation}
    \sup_{\theta\in U}
    \norm{
        \nabla_\theta^2\Phi_{N,n}(\theta)
        -
        \nabla_\theta^2\Phi_N(\theta)
    }_{\mathrm{HS}}
    \convp0 .
\end{equation}
    \end{itemize}
\end{lemma}

\begin{proof}
By \eqref{eq:discrete-score} and \eqref{eq:first-derivative-h} , for
\(\theta\in U\),
\begin{align}
\nabla_\theta\Phi_{N,n}(\theta)-\nabla_\theta\Phi_N(\theta)
&=
    \int_0^T
        \{\nabla_\theta h_n(t,\theta)-\nabla_\theta h(t,\theta)\}
        \{d\Ybar_t-h(t,\theta)\,dt\} \notag\\
&\quad+
    \int_0^T
        \nabla_\theta h_n(t,\theta)
        \{h(t,\theta)-h_n(t,\theta)\}\,dt .
    \label{eq:discrete-local-score-decomposition}
\end{align}
From \eqref{eq:discrete-local-score-step-rate}, finite-variation terms in the above decomposition are bounded by \(C\Delta_n\). Applying
\zcref{lem:StochIntVanishesCombined} componentwise gives
\begin{equation}\label{eq:discrete-local-score-martingale-bound}
    \sqrt{N}
    \sup_{\theta\in U}
    \norm{
        \int_0^T
        \{\nabla_\theta h_n(t,\theta)-\nabla_\theta h(t,\theta)\}
        \cdot d\Wbar_t
    }_{\mathbb{R}^p}
    \convp0 .
\end{equation}
Together with \(\sqrt{N}\Delta_n\to0\), this proves (i).
\begingroup
For (ii), using the displayed formulas for the continuous and discrete
Hessians, for \(\theta\in U\),
\begin{align*}
&\nabla_\theta^2\Phi_{N,n}(\theta)-\nabla_\theta^2\Phi_N(\theta) \notag\\
&\quad=
\int_0^T
\{\nabla_\theta^2h_n(t,\theta)-\nabla_\theta^2h(t,\theta)\}
\cdot\{d\Ybar_t-h(t,\theta)\,dt\} \notag\\
&\qquad+
\int_0^T \nabla_\theta^2h_n(t,\theta)
\cdot\{h(t,\theta)-h_n(t,\theta)\}\,dt \notag\\
&\qquad-
\int_0^T
\left\{
\nabla_\theta h_n(t,\theta)\nabla_\theta h_n(t,\theta)^\top
-
\nabla_\theta h(t,\theta)\nabla_\theta h(t,\theta)^\top
\right\}\,dt .
\end{align*}
All finite-variation terms are \(O(\Delta_n)\), uniformly on \(U\), by
\eqref{eq:discrete-local-score-step-rate} and boundedness of the relevant
derivatives. For the martingale term, apply
\zcref{lem:StochIntVanishesCombined}(i), componentwise, to
\[
f_n(t,\theta)=\nabla_\theta^2h_n(t,\theta)-\nabla_\theta^2h(t,\theta).
\]
Under \(\ell=3\), both \(f_n\) and \(\partial_\theta f_n\) are bounded by
\(C\Delta_n\). Hence the martingale term is \(o_P(1)\), uniformly on \(U\),
and (ii) follows.
\endgroup
\end{proof}
\begin{proof}[Proof of \zcref{thm:discrete-asymptotic-normality}]
It is enough to argue along an arbitrary sequence \((N,n)\) satisfying
\(N\to\infty\), \(n\to\infty\), and
\(\sqrt{N}\Delta_n\to0\). Along such a sequence, we write
\(n=n(N)\). Since
\(\theta_0\in\Theta^\circ\), we have
\(\hat\theta_{N,n}\in\Theta^\circ\) eventually almost surely. Taylor's theorem gives
\begin{align}
0
&=
\sqrt{N}\,\nabla_\theta\Phi_{N,n}(\theta_0)+
\left\{
\int_0^1
\nabla_\theta^2\Phi_{N,n}
\bigl(
\theta_0+s(\hat\theta_{N,n}-\theta_0)
\bigr)\,ds
\right\}
\sqrt{N}(\hat\theta_{N,n}-\theta_0).
\end{align}

By \zcref{lem:discrete-local-hessian-comparison}(i), we have
\begin{equation}
\sqrt{N}
\left\{
\nabla_\theta\Phi_{N,n}(\theta_0)
-
\nabla_\theta\Phi_N(\theta_0)
\right\}
\convp0.
\end{equation}
Together with \zcref{prop:ScoreCLT}, this yields
\begin{equation}
\sqrt{N}\,\nabla_\theta\Phi_{N,n}(\theta_0)
\xrightarrow{d}
\mathcal{N}(0,\Sigma).
\end{equation}

Moreover,
\begin{equation}
\sup_{0\le s\le1}
\norm{
\theta_0+s(\hat\theta_{N,n}-\theta_0)-\theta_0
}_{\mathbb{R}^p}
\le
\norm{\hat\theta_{N,n}-\theta_0}_{\mathbb{R}^p}
\convp0.
\end{equation}
Hence, by \zcref{lem:discrete-local-hessian-comparison}(ii) and
the same argument used to prove \eqref{HessianConvergence},
\begin{equation}
\int_0^1
\nabla_\theta^2\Phi_{N,n}
\bigl(
\theta_0+s(\hat\theta_{N,n}-\theta_0)
\bigr)\,ds
\convp
-I(\theta_0).
\end{equation}
Since \(I(\theta_0)\) is positive definite, Slutsky's theorem gives
\begin{equation}
\sqrt{N}(\hat\theta_{N,n}-\theta_0)
\xrightarrow{d}
\mathcal{N}\left(
0,
I(\theta_0)^{-1}\Sigma I(\theta_0)^{-1}
\right).
\end{equation}
This completes the proof.
\end{proof}

\subsubsection{Proof of \zcref{thm:discrete-covariance-estimator-consistency}}
\label{subsubsec:proof-discrete-covariance-consistency}

For \(\theta\in\Theta\), put
\begin{equation}\label{eq:discrete-continuous-mean-second-definitions}
    m_{N,n}(\theta)
    :=
    \frac{1}{N}\sum_{i=1}^N \eta_{i,n}(\theta),
    \qquad
    S_{N,n}(\theta)
    :=
    \frac{1}{N}\sum_{i=1}^N
    \eta_{i,n}(\theta)\eta_{i,n}(\theta)^\top,
\end{equation}
We first compare the discrete score contributions with their continuous
counterparts.
\begin{proposition}
\label{prop:discrete-score-contribution-comparison}
Suppose that \zcref{ass:parameter-space} and \zcref{differentiability}(ii)
hold. Then as \(N\to\infty\) and
\(n\to\infty\),
\begin{equation}\label{eq:discrete-mean-comparison-covariance-proof}
    \sup_{\theta\in\Theta}
    \norm{m_{N,n}(\theta)-m_N(\theta)}_{\mathbb{R}^p}
    \convas0,
\qquad
    \sup_{\theta\in\Theta}
    \norm{S_{N,n}(\theta)-S_N(\theta)}_{\mathrm{HS}}
    \convas0.
\end{equation}
\end{proposition}

\begin{proof}
From \eqref{eq:eta-i-theta-definition} and \eqref{eq:discrete-eta-i-theta}, for each \(i\),
\begin{align}
\delta_{i,n}(\theta):&=
    \eta_{i,n}(\theta)-\eta_i(\theta)\\
&=
    \int_0^T
        \{\nabla_\theta h_n(t,\theta)-\nabla_\theta h(t,\theta)\}
        \{d\Yvec[i]_t-h(t,\theta)\,dt\}
    \notag\\
&\quad+
    \int_0^T
        \nabla_\theta h_n(t,\theta)
        \{h(t,\theta)-h_n(t,\theta)\}\,dt .
    \label{eq:delta-i-n-decomposition}
\end{align}
From \eqref{eq:discrete-local-score-step-rate}, finite-variation terms in the above decomposition are bounded by \(C\Delta_n\), uniformly in \(i\) and \(\theta\). Since \(\Delta_n=T/n\),
\zcref{lem:StochIntVanishesCombined}, applied componentwise with \(a_n=C\Delta_n\), gives
\begin{equation}
    \lim_{K\to\infty}
    \sup_{\substack{N\ge K\\ n\ge K}}
    \frac{1}{N}\sum_{i=1}^N
    \sup_{\theta\in\Theta}
    \norm{\int_0^T
        \{\nabla_\theta h_n(t,\theta)-\nabla_\theta h(t,\theta)\}
        \cdot d\Wvec[i]_t}_{\mathbb{R}^p}^2
    =0
    \quad\text{a.s.}
\end{equation}
Therefore,
\begin{equation}
\begin{aligned}
    \sup_{\theta\in\Theta}
    \norm{m_{N,n}(\theta)-m_N(\theta)}_{\mathbb{R}^p}
    &\le
    \left(
        \sup_{\theta\in\Theta}
        \frac{1}{N}\sum_{i=1}^N
        \norm{\delta_{i,n}(\theta)}_{\mathbb{R}^p}^2
    \right)^{1/2}
    \convas0.
\end{aligned}
\end{equation}
Furthermore,
\begin{align}
&S_{N,n}(\theta)-S_N(\theta) \notag\\
&\quad=
    \frac{1}{N}\sum_{i=1}^N
    \left\{
        \delta_{i,n}(\theta)\eta_i(\theta)^\top
        +\eta_i(\theta)\delta_{i,n}(\theta)^\top
        +\delta_{i,n}(\theta)\delta_{i,n}(\theta)^\top
    \right\}.
\end{align}
The last term converges to \(0\) almost surely by the preceding estimate. By the Cauchy--Schwarz inequality,
\begin{align}
&\sup_{\theta\in\Theta}
\norm{
    \frac{1}{N}\sum_{i=1}^N
    \delta_{i,n}(\theta)\eta_i(\theta)^\top
}_{\mathrm{HS}}\le
\left(
    \sup_{\theta\in\Theta}
    \frac{1}{N}\sum_{i=1}^N
    \norm{\delta_{i,n}(\theta)}_{\mathbb{R}^p}^2
\right)^{1/2}
\left(
    \sup_{\theta\in\Theta}
    \frac{1}{N}\sum_{i=1}^N
    \norm{\eta_i(\theta)}_{\mathbb{R}^p}^2
\right)^{1/2}.
    \label{eq:cross-term-cauchy-schwarz}
\end{align}
\begingroup
The first factor on the right-hand side converges to \(0\) almost surely by
the preceding estimate. By \zcref{lem:SecondMomentUniformLLN} and the continuity of \(S\) on
the compact set \(\Theta\), the second factor is almost surely bounded for
all sufficiently large \(N\). Thus the cross term converges to zero almost
surely. Its transpose is handled identically, while the
\(\delta_{i,n}\delta_{i,n}^\top\) term converges to zero by the preceding
estimate. This completes the proof.
\endgroup
\end{proof}

\begin{proof}[Proof of \zcref{thm:discrete-covariance-estimator-consistency}]
We prove the asserted almost sure convergence in the double-index sense
\(N\to\infty\), \(n\to\infty\).
By the definition of $\hat\Sigma_{N,n}$,
\begin{equation}\label{eq:discrete-covariance-decomposition}
    \hat\Sigma_{N,n}
    =
    S_{N,n}(\hat\theta_{N,n})
    -m_{N,n}(\hat\theta_{N,n})m_{N,n}(\hat\theta_{N,n})^\top .
\end{equation}
Combining \zcref{prop:discrete-score-contribution-comparison} and the proof of \zcref{thm:covariance-estimator-consistency},
\begin{equation}
    \sup_{\theta\in\Theta}\norm{m_{N,n}(\theta)-m(\theta)}_{\mathbb{R}^p}\convas0.
\end{equation}
By \zcref{thm:discrete-consistency} and \(m(\theta_0)=0\), we obtain
\begin{equation}\label{eq:discrete-covariance-mean-vanishes}
    m_{N,n}(\hat\theta_{N,n})
    \convas0 .
\end{equation}
By \zcref{prop:discrete-score-contribution-comparison} and \zcref{lem:SecondMomentUniformLLN},
\begin{equation}\label{eq:discrete-covariance-second-comparison}
    \sup_{\theta\in\Theta}
    \norm{S_{N,n}(\theta)-S(\theta)}_{\mathrm{HS}}
    \convas0,
\end{equation}
and we have
\begin{equation}\label{eq:discrete-covariance-second-moment-converges}
    S_{N,n}(\hat\theta_{N,n})
    \convas S(\theta_0)=\Sigma .
\end{equation}
Therefore,
\begin{equation}
    \hat\Sigma_{N,n}\convas\Sigma .
\end{equation}
It remains to prove $\hat V_{N,n}\convas I(\theta_0)^{-1}\Sigma I(\theta_0)^{-1}$. Note that
\begin{align}
&\norm{
    \nabla_\theta h_n(t,\theta)\nabla_\theta h_n(t,\theta)^\top
    -
    \nabla_\theta h(t,\theta)\nabla_\theta h(t,\theta)^\top
}_{\mathrm{HS}} \lesssim
\norm{\nabla_\theta h_n(t,\theta)-\nabla_\theta h(t,\theta)}_{\mathrm{HS}}.
\end{align}
By \eqref{eq:discrete-local-score-step-rate}, we have
\begin{equation}
\sup_{\theta\in\Theta}
\norm{
    \int_0^T
        \nabla_\theta h_n(t,\theta)
        \nabla_\theta h_n(t,\theta)^\top\,dt
    -
    \int_0^T
        \nabla_\theta h(t,\theta)
        \nabla_\theta h(t,\theta)^\top\,dt
}_{\mathrm{HS}}
\to0 .
\end{equation}
Together with \zcref{thm:discrete-consistency}, this yields
\begin{equation}
    \hat I_{N,n}\convas I(\theta_0).
\end{equation}
\begingroup
Since \(I(\theta_0)\) is positive definite, matrix inversion is continuous
on a neighbourhood of \(I(\theta_0)\), so
\(\hat I_{N,n}^{-1}\convas I(\theta_0)^{-1}\). Consequently,
\endgroup
\begin{equation}
    \hat V_{N,n}
    =
    \hat I_{N,n}^{-1}\hat\Sigma_{N,n}\hat I_{N,n}^{-1}
    \convas
    I(\theta_0)^{-1}\Sigma I(\theta_0)^{-1}.
\end{equation}
This completes the proof of \zcref{thm:discrete-covariance-estimator-consistency}.
\end{proof}

\endgroup

\section{Numerical example}
\label{sec:numerical-example}

This section is devoted to a numerical examination of the theoretical results
for a continuously monitored two-level system.  We first specify an example
system for which the deterministic observation drift is available explicitly
through the Bloch representation.  We then describe the simulation design for
the sample paths and the numerical maximization of the contrast.  Finally, a
Monte Carlo study examines the consistency of the parameter and covariance
estimators, followed by the marginal and joint distributional behavior relevant
to asymptotic normality and the mesh-balance condition.

\subsection{Example system}
\label{subsec:numerical-model-implementation}

Let
\[
    \sigma_x=
    \begin{pmatrix}
        0&1\\
        1&0
    \end{pmatrix},
    \qquad
    \sigma_y=
    \begin{pmatrix}
        0&-i\\
        i&0
    \end{pmatrix},
    \qquad
    \sigma_z=
    \begin{pmatrix}
        1&0\\
        0&-1
    \end{pmatrix}.
\]
We consider the one-dimensional observation case \(r=1\) and set
\begin{equation}\label{eq:numerical-model-HL}
    H^\alpha=\frac{\alpha}{2}\sigma_x,
    \qquad
    L^\beta=\beta\sigma_z,
    \qquad
    \theta=(\alpha,\beta)^\top.
\end{equation}
Writing the averaged density matrix in Bloch form as
\begin{equation}
    \rhoave{t,\theta}
    =\frac{1}{2}
    \left(I+x_t^\theta\sigma_x+y_t^\theta\sigma_y+z_t^\theta\sigma_z\right),
\end{equation}
the deterministic averaged equation reduces to
\begin{equation}\label{eq:numerical-averaged-bloch}
    \dot{x}_t^\theta=-2\beta^2x_t^\theta,
    \qquad
    \dot{y}_t^\theta=-2\beta^2y_t^\theta-\alpha z_t^\theta,
    \qquad
    \dot{z}_t^\theta=\alpha y_t^\theta.
\end{equation}
The corresponding deterministic observation drift is
\begin{equation}\label{eq:numerical-observation-drift}
    h(t,\theta)
    =
    \Tr\!\left(
        L^\beta\rhoave{t,\theta}
        +\rhoave{t,\theta}(L^\beta)^*
    \right)
    =
    2\beta z_t^\theta.
\end{equation}

For the numerical experiment we take
\begin{equation}\label{eq:numerical-initial-state}
    \rho_0=\frac{I+\sigma_z}{2},
    \qquad
    x_0^\theta=y_0^\theta=0,
    \qquad
    z_0^\theta=1.
\end{equation}
The \((y,z)\)-subsystem is linear with constant coefficients, and therefore
\begin{equation}\label{eq:numerical-exact-bloch-flow}
    \begin{pmatrix}y_t^\theta\\ z_t^\theta\end{pmatrix}
    =
    \exp\!\left[
        t\begin{pmatrix}
            -2\beta^2 & -\alpha\\
            \alpha & 0
        \end{pmatrix}
    \right]
    \begin{pmatrix}0\\1\end{pmatrix}.
\end{equation}
Thus \(h(t,\theta)\) is evaluated directly from
\eqref{eq:numerical-exact-bloch-flow}, and in the present scalar model the
discrete contrast \(\Phi_{N,n}\) in \eqref{eq:discrete-contrast} becomes
\begin{equation}\label{eq:numerical-explicit-discrete-contrast}
    \Phi_{N,n}(\theta)
    =
    2\beta\sum_{j=1}^{n}z_{t_{j-1}}^\theta\,\Delta_j\Ybar
    -2\beta^2\Delta_n\sum_{j=1}^{n}
        \bigl(z_{t_{j-1}}^\theta\bigr)^2.
\end{equation}

\begin{numericalremark}\label{rem:numerical-alpha-symmetry}
For the initial state \eqref{eq:numerical-initial-state}, the \(z\)-component
satisfies
\begin{equation}\label{eq:numerical-z-second-order}
    \ddot z_t^\theta+2\beta^2\dot z_t^\theta+\alpha^2 z_t^\theta=0,
    \qquad
    z_0^\theta=1,
    \qquad
    \dot z_0^\theta=0.
\end{equation}
Hence
\begin{equation}\label{eq:numerical-alpha-sign-symmetry}
    z_t^{(-\alpha,\beta)}=z_t^{(\alpha,\beta)},
    \qquad
    h(t,-\alpha,\beta)=h(t,\alpha,\beta),
\end{equation}
and the contrast is even in \(\alpha\).  This symmetry determines the
identifiable branch used in the numerical optimization below.

Moreover, for any true parameter \(\theta_0=(\alpha_0,\beta_0)\) with
\(\beta_0\neq0\), the model is identifiable on the branch \(\alpha\geq0\),
with \(\beta\) allowed to range over \(\mathbb{R}\).  Indeed,
\begin{equation}
    h(0,\alpha,\beta)=2\beta,
    \qquad
    \partial_t h(0,\alpha,\beta)=0,
    \qquad
    \partial_t^2 h(0,\alpha,\beta)=-2\beta\alpha^2.
\end{equation}
Thus \(h(t,\theta)=h(t,\theta_0)\) for all \(t\) first gives
\(\beta=\beta_0\), and then, because \(\beta_0\neq0\), comparison of the
second derivatives gives \(\alpha^2=\alpha_0^2\).  Hence
\(\alpha=\alpha_0\) on \(\alpha\geq0\).  By contrast, if
\(\beta_0=0\), then \(h(t,\alpha,0)\equiv0\) for every \(\alpha\), so
\(\alpha\) is not identifiable.
\end{numericalremark}

The observation paths used in the Monte Carlo experiment are generated from
the stochastic Bloch equation
\begin{equation}\label{eq:numerical-stochastic-bloch}
\left\{
\begin{aligned}
    dX_t
    &=
    -2\beta_0^2X_t\,dt
    -2\beta_0Z_tX_t\,dW_t,\\
    dY_t
    &=
    \left(-2\beta_0^2Y_t-\alpha_0Z_t\right)dt
    -2\beta_0Z_tY_t\,dW_t,\\
    dZ_t
    &=
    \alpha_0Y_t\,dt
    +2\beta_0(1-Z_t^2)\,dW_t,
\end{aligned}
\right.
\end{equation}
with observation equation
\begin{equation}\label{eq:numerical-observation-equation}
    dY_t^{\mathrm{obs}}
    =
    2\beta_0 Z_t\,dt+dW_t.
\end{equation}
Here \(X_t,Y_t,Z_t\) are the stochastic Bloch coordinates, whereas
\(Y_t^{\mathrm{obs}}\) denotes the scalar observation process.

\begin{numericalremark}\label{rem:numerical-measurement-strength}
We briefly comment on the physical interpretation of \(\beta\).  In the
continuous-measurement convention of Jacobs and
Steck~\cite{jacobs2006straightforward}, a Hermitian measurement operator is
written in the form \(c=\sqrt{2k}\,X\), where \(k\) is the measurement
strength.  Thus, for \(L^\beta=\beta\sigma_z\), the corresponding strength is
\(k=\beta^2/2\).
\end{numericalremark}

\subsection{Simulation design for sample paths}
\label{subsec:numerical-simulation-design}

We use the model above to examine first the consistency results for
discretely observed data and then the asymptotic normality result, including
the role of the stronger condition \(\sqrt{N}\Delta_n\to0\).  We set
\begin{equation}\label{eq:numerical-true-parameter-setting}
    \theta_0=(\alpha_0,\beta_0)=(1.2,0.7),
    \qquad
    T=5,
    \qquad
    \rho_0=\frac{I+\sigma_z}{2},
\end{equation}
and consider the nine combinations
\begin{equation}\label{eq:numerical-mesh-design}
    N\in\{4,100,1000\},
    \qquad
    \Delta_n\in\{0.5,0.1,0.01\}.
\end{equation}
For each combination we use \(1000\) Monte Carlo replications and maximize
over the fixed compact parameter space \(\Theta=[0,10]\times[0,10]\).

\begin{numericalremark}\label{rem:numerical-parameter-space}
We take \(\alpha\geq0\) because
\(h(t,-\alpha,\beta)=h(t,\alpha,\beta)\) for the present initial state, so this
fixes the identifiable branch.  As noted in \zcref[S]{rem:numerical-alpha-symmetry},
identifiability itself does not require \(\beta\geq0\) when \(\beta_0\neq0\).
We nevertheless restrict \(\beta\geq0\) in view of the physical interpretation
of \(\beta\) discussed in \zcref[S]{rem:numerical-measurement-strength}. The upper
bound \(10\) is imposed only to obtain a compact search region for this numerical
experiment.
\end{numericalremark}

The stochastic Bloch equation and the observation process are simulated on a
finer grid with step \(\delta=10^{-3}\).  Let
\(r_m=(x_m,y_m,z_m)^\top\) denote the simulated Bloch vector and let
\(\widetilde r_{m+1}\) be the value obtained from one Euler--Maruyama step of
the Bloch SDE before enforcing the state-space constraint.  To prevent a
finite-step Euler--Maruyama update from leaving the Bloch ball, we use the
radial projection
\begin{equation}\label{eq:numerical-bloch-ball-projection}
    r_{m+1}
    =
    \Pi_{\mathbb B^3}(\widetilde r_{m+1})
    :=
    \begin{cases}
        \widetilde r_{m+1},
        & \|\widetilde r_{m+1}\|_2\leq 1,\\[2mm]
        \displaystyle
        \frac{\widetilde r_{m+1}}
             {\|\widetilde r_{m+1}\|_2},
        & \|\widetilde r_{m+1}\|_2>1.
    \end{cases}
\end{equation}
The observation increment is evaluated from the pre-update value \(z_m\), as
\begin{equation}\label{eq:numerical-fine-observation-increment}
    \Delta_m Y
    =2\beta_0 z_m\,\delta+\Delta_m W,
    \qquad
    \Delta_m W\sim\mathcal N(0,\delta),
\end{equation}
and the fine-grid increments are then aggregated to the three observation grids
in \eqref{eq:numerical-mesh-design}.

\begin{numericalremark}
The use of a finer grid for data generation avoids using the same coarse
discretization both to generate the observations and to evaluate the contrast.
Otherwise, matching discretization errors could make the numerical agreement
artificially favorable.  The projection in
\eqref{eq:numerical-bloch-ball-projection} is used only in the stochastic
Euler--Maruyama data-generation step.  The deterministic function
\(h(t,\theta)\) entering the contrast is evaluated from the exact averaged
Bloch dynamics described above.
\end{numericalremark}

Within each Monte Carlo replication, the same fine-grid Brownian paths are used
for all three observation grids.  The samples for \(N=4,100,1000\) are nested.

The numerical maximization of the contrast is carried out as follows.
\begin{enumerate}
    \item The contrast is evaluated on a uniform grid over the whole parameter
    space \(\Theta=[0,10]\times[0,10]\), with grid spacing \(0.05\) in both
    coordinates.

    \item A grid point is regarded as a local-maximum candidate when its
    contrast value is no smaller than those at all neighboring grid points
    (up to eight neighbors).  The candidates are ordered by their contrast
    values, and the eight highest well-separated candidates are selected as
    starting points, with a minimum Euclidean separation of \(0.15\) between
    selected points.

    \item Starting from each selected grid point, the contrast is refined by a
    bounded Gauss--Newton-type continuous optimization with backtracking, with
    the iterates constrained to remain in \(\Theta\).

    \item The refined candidates are compared, and the one with the largest
    contrast value is taken as \(\hat\theta_{N,n}\).
\end{enumerate}
Thus the grid is used to locate several possible maxima over the full parameter
space, while the final estimate is obtained after continuous refinement.

\subsection{Monte Carlo study}
\label{subsec:numerical-monte-carlo-study}

\subsubsection{Monte Carlo means}
\label{subsubsec:numerical-monte-carlo-means}

For each \((N,\Delta_n)\) of the nine designs, \zcref[S]{tab:numerical-discrete-mesh-consistency} reports
the Monte Carlo means of the unscaled estimators \(\hat\alpha_{N,n}\) and
\(\hat\beta_{N,n}\).  The empirical standard deviation over the \(1000\)
replications is shown in parentheses below each mean.  This first comparison
concerns the consistency result in \zcref[S]{thm:discrete-consistency}, before
introducing the \(\sqrt N\)-scaled statistics used for asymptotic normality.

\begin{table}
\caption{Monte Carlo means of the parameter estimators based on \(1000\)
replications: (a) \(\hat\alpha_{N,n}\) and
(b) \(\hat\beta_{N,n}\).  The empirical standard deviation is shown in
parentheses below each mean.  The true parameter is
\(\theta_0=(\alpha_0,\beta_0)=(1.2,0.7)\).}
\label{tab:numerical-discrete-mesh-consistency}
\begin{minipage}[t]{0.49\textwidth}
    \centering
    \textbf{(a) \(\alpha\)}\\[3pt]
    \setlength{\tabcolsep}{3.5pt}
    \renewcommand{\arraystretch}{1.08}
    \begin{tabular}{@{}lccc@{}}
\toprule
$N$ & $\Delta_n=0.5$ & $\Delta_n=0.1$ & $\Delta_n=0.01$ \\
\midrule
4 & $1.790$ & $1.819$ & $1.728$ \\
 & $(1.801)$ & $(1.956)$ & $(1.840)$ \\[.2em]
100 & $1.278$ & $1.223$ & $1.210$ \\
 & $(0.146)$ & $(0.134)$ & $(0.131)$ \\[.2em]
1000 & $1.272$ & $1.220$ & $1.207$ \\
 & $(0.045)$ & $(0.042)$ & $(0.042)$ \\
\bottomrule
\end{tabular}

\end{minipage}\hfill
\begin{minipage}[t]{0.49\textwidth}
    \centering
    \textbf{(b) \(\beta\)}\\[3pt]
    \setlength{\tabcolsep}{3.5pt}
    \renewcommand{\arraystretch}{1.08}
    \begin{tabular}{@{}lccc@{}}
\toprule
$N$ & $\Delta_n=0.5$ & $\Delta_n=0.1$ & $\Delta_n=0.01$ \\
\midrule
4 & $0.608$ & $0.657$ & $0.668$ \\
 & $(0.251)$ & $(0.313)$ & $(0.336)$ \\[.2em]
100 & $0.641$ & $0.686$ & $0.695$ \\
 & $(0.076)$ & $(0.079)$ & $(0.080)$ \\[.2em]
1000 & $0.649$ & $0.694$ & $0.703$ \\
 & $(0.026)$ & $(0.026)$ & $(0.027)$ \\
\bottomrule
\end{tabular}

\end{minipage}
\end{table}

Along the diagonal refinement from \((N,\Delta_n)=(4,0.5)\) to
\((100,0.1)\) and then \((1000,0.01)\), the Monte Carlo means move toward
\((\alpha_0,\beta_0)=(1.2,0.7)\), while the empirical standard deviations
contract substantially.  The table also separates the effects of increasing
\(N\) and refining the observation mesh.  For example, at the fixed coarse
mesh \(\Delta_n=0.5\), increasing \(N\) sharply reduces the dispersion but
does not remove the visible discretization bias in the \(\alpha\)-coordinate.
This is consistent with the joint limit \(N\to\infty\) and
\(n\to\infty\) in \zcref[S]{thm:discrete-consistency}.

The covariance-consistency result in
\zcref[S]{thm:discrete-covariance-estimator-consistency} has the limiting target
\begin{equation}\label{eq:numerical-limiting-covariance}
    V
    :=
    I(\theta_0)^{-1}\Sigma I(\theta_0)^{-1}.
\end{equation}
For the present model, the expectation defining \(\Sigma\), and hence the
numerical value of \(V\), is not available in closed form.  We therefore
replace \(V\) in the numerical diagnostics by a fine-grid Monte Carlo
approximation.  On a grid with spacing \(\Delta t\), we first approximate
\(I(\theta_0)\) by
\begin{equation}\label{eq:numerical-fine-grid-information}
    I^{(\Delta t)}
    =
    \Delta t\sum_{k=1}^{T/\Delta t}
    \nabla_\theta h(t_{k-1},\theta_0)
    \nabla_\theta h(t_{k-1},\theta_0)^\top.
\end{equation}
For each independently generated single trajectory, we form the fine-grid
score approximation
\begin{equation}\label{eq:numerical-fine-grid-score}
    \eta_i^{(\Delta t)}
    =
    \sum_{k=1}^{T/\Delta t}
    \nabla_\theta h(t_{k-1},\theta_0)
    \left\{
        \Delta_k\Yvec[i]
        -h(t_{k-1},\theta_0)\Delta t
    \right\}.
\end{equation}
The expectation in \(\Sigma\) is then approximated by the empirical second
moment of these scores, and the resulting approximation of \(V\) is obtained
from the same sandwich formula as in
\eqref{eq:numerical-limiting-covariance}.  Using \(25000\) independent
single-trajectory Monte Carlo paths with \(\Delta t=10^{-3}\) gives
\begin{equation}\label{eq:numerical-limiting-covariance-value}
    V
    \approx
    \begin{pmatrix}
        1.6805 & 0.6367\\
        0.6367 & 0.6476
    \end{pmatrix}.
\end{equation}

We also checked the sensitivity of this approximation to the fine-grid
spacing.  In a separate paired-grid calculation with \(25000\) trajectories,
we used a grid with \(\Delta t=10^{-4}\) and drove the corresponding
\(\Delta t=10^{-3}\) Euler--Maruyama scheme by aggregating ten consecutive
Brownian increments from the same paths.  The maximum relative change among
the \((1,1)\), \((1,2)\), and \((2,2)\) entries was approximately
\(0.17\%\).  Thus the fine-grid discretization effect is already well below
the \(10^{-2}\) scale for all three displayed covariance entries at
\(\Delta t=10^{-3}\).

To examine directly the covariance estimator in
\zcref[S]{thm:discrete-covariance-estimator-consistency},
\zcref[S]{tab:numerical-covariance-mesh-comparison} reports the Monte Carlo means
of the entries of \(\hat V_{N,n}\) over the \(1000\) replications.  For a
replication in which \(\hat I_{N,n}\) is singular, \(\hat V_{N,n}\) is
undefined and that replication is omitted from the corresponding mean.

\begin{table}
\caption{Monte Carlo means of the defined entries of the empirical sandwich
covariance estimator \(\hat V_{N,n}\).  The last column gives the corresponding
entry of the fine-grid approximation of \(V\), shown once per entry block.  \(V\)
does not depend on \(N\) or \(\Delta_n\).}
\label{tab:numerical-covariance-mesh-comparison}
\centering
\setlength{\tabcolsep}{7pt}
\renewcommand{\arraystretch}{1.08}
\begin{tabular}{llcccr}
\toprule
Entry & \(N\) & \(\Delta_n=0.5\) & \(\Delta_n=0.1\) & \(\Delta_n=0.01\) & \(V\)\\
\midrule
\((1,1)\) & 4    & 13.2047 & 1.2430 & 1.4616 & \\
           & 100  & 1.2018  & 1.5470 & 1.6482 & \\
           & 1000 & 1.2219  & 1.5824 & 1.6852 & 1.6805\\
\midrule
\((1,2)\) & 4    & 0.1966 & 0.3793 & 0.4571 & \\
           & 100  & 0.3004 & 0.5379 & 0.6120 & \\
           & 1000 & 0.3240 & 0.5659 & 0.6408 & 0.6367\\
\midrule
\((2,2)\) & 4    & 0.3366 & 0.4435 & 0.4799 & \\
           & 100  & 0.4762 & 0.5918 & 0.6323 & \\
           & 1000 & 0.4877 & 0.6064 & 0.6475 & 0.6476\\
\bottomrule
\end{tabular}
\end{table}

For \((N,\Delta_n)=(1000,0.01)\), the Monte Carlo mean of the empirical
covariance estimator is approximately
\[
    \begin{pmatrix}
        1.6852 & 0.6408\\
        0.6408 & 0.6475
    \end{pmatrix}.
\]
The three displayed entries are close to the fine-grid reference values in
\eqref{eq:numerical-limiting-covariance-value}.  The table also shows the
finite-sample instability of covariance estimation for \(N=4\).  In
particular, the relatively large Monte Carlo mean of the \((1,1)\) entry at
\(\Delta_n=0.5\) is caused by nearly singular plug-in information matrices in
some replications.  This is consistent with the instability of the \(N=4\)
Studentized distributions discussed below.

The convergence in \zcref[S]{tab:numerical-covariance-mesh-comparison} is the
numerical counterpart of the covariance-consistency result: as the number of
paths increases and the observation mesh is refined, the empirical sandwich
covariance approaches the same limiting matrix \(V\) that enters the
asymptotic normal law.  The Studentized comparisons below provide a separate
check of the practical effect of replacing \(V\) by \(\hat V_{N,n}\).

Having first examined consistency, we next turn to the distributional
approximation in \zcref[S]{thm:discrete-asymptotic-normality}.  In addition to
\(N\to\infty\) and \(n\to\infty\), this result requires the stronger mesh
balance condition \(\sqrt{N}\Delta_n\to0\).  We examine this through the
marginal and joint distributional comparisons below.

\subsubsection{Histograms and QQ plots for marginal distributions}
\label{subsubsec:numerical-marginal-distributions}

Using the limiting covariance \(V\) as the theoretical reference, define
\begin{equation}\label{eq:numerical-standardized-mesh-error}
    Z_{N,n}^{\alpha}
    :=
    \frac{\sqrt{N}(\hat\alpha_{N,n}-\alpha_0)}
    {\sqrt{V_{11}}},
    \qquad
    Z_{N,n}^{\beta}
    :=
    \frac{\sqrt{N}(\hat\beta_{N,n}-\beta_0)}
    {\sqrt{V_{22}}}.
\end{equation}
For comparison, whenever the empirical sandwich covariance
\(\hat V_{N,n}\) in \eqref{eq:discrete-V-hat} is defined, define the
corresponding Studentized errors by
\begin{equation}\label{eq:numerical-studentized-mesh-error}
    T_{N,n}^{\alpha}
    :=
    \frac{\sqrt{N}(\hat\alpha_{N,n}-\alpha_0)}
    {\sqrt{(\hat V_{N,n})_{11}}},
    \qquad
    T_{N,n}^{\beta}
    :=
    \frac{\sqrt{N}(\hat\beta_{N,n}-\beta_0)}
    {\sqrt{(\hat V_{N,n})_{22}}}.
\end{equation}
The empirical covariance \(\hat V_{N,n}\) is computed from the same observed
trajectories as the corresponding parameter estimate, using
\eqref{eq:discrete-Sigma-hat}--\eqref{eq:discrete-V-hat}.

For the marginal histograms, the common display range \([-3,3]\) is divided
into \(18\) equal-width bins.  Thus, with
\begin{equation}\label{eq:numerical-histogram-bins}
    b_k=-3+\frac{k}{3},
    \qquad k=0,\ldots,18,
\end{equation}
the \(k\)-th bin is \([b_{k-1},b_k)\), except that the right endpoint is
included in the last bin.  For a sample \(X_1,\ldots,X_m\), the plotted bar
height is
\begin{equation}\label{eq:numerical-histogram-height}
    H_k(X)
    =
    \frac{1}{m(b_k-b_{k-1})}
    \sum_{r=1}^{m}
    \mathbf 1_{[b_{k-1},b_k)}(X_r),
    \qquad k=1,\ldots,18,
\end{equation}
with the evident right-endpoint convention for the last bin.  Here \(m=1000\)
for the \(Z\)-statistics, while for a Studentized statistic \(m\) is the number
of replications for which the corresponding \(\hat V_{N,n}\) is defined.
Consequently, observations outside \([-3,3]\) are not discarded before
normalization: the total area of the displayed bars equals the empirical
fraction of the relevant sample lying in \([-3,3]\).  In particular, the
histograms are not renormalized over the displayed range, so their heights can
be compared directly with the standard normal density shown in the same panel.

\begin{figure}
\centering
\begin{minipage}[t]{0.49\textwidth}
    \centering
    \textbf{(a) Histogram: \(\alpha\)}\\[2pt]
    \includegraphics[width=\textwidth]{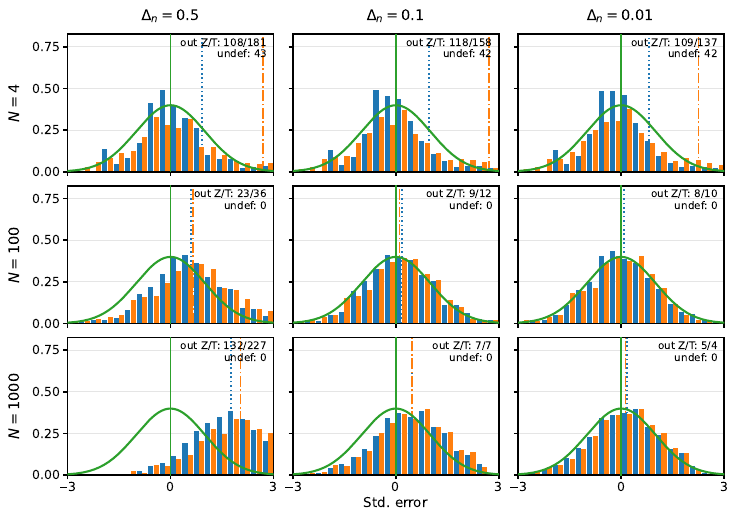}
\end{minipage}\hfill
\begin{minipage}[t]{0.49\textwidth}
    \centering
    \textbf{(b) Histogram: \(\beta\)}\\[2pt]
    \includegraphics[width=\textwidth]{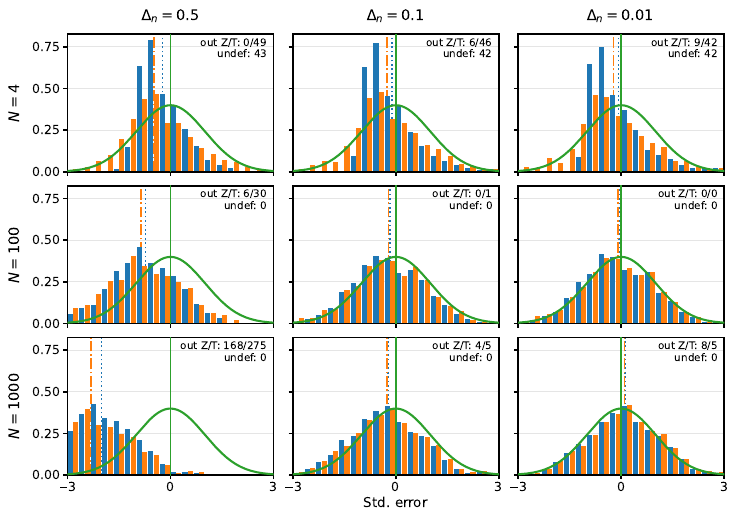}
\end{minipage}

\vspace{0.6em}

\begin{minipage}[t]{0.49\textwidth}
    \centering
    \textbf{(c) Normal QQ plot: \(\alpha\)}\\[2pt]
    \includegraphics[width=\textwidth]{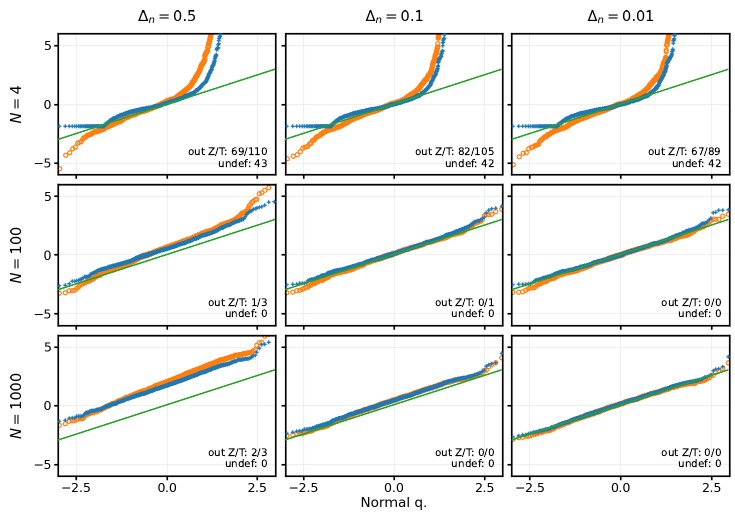}
\end{minipage}\hfill
\begin{minipage}[t]{0.49\textwidth}
    \centering
    \textbf{(d) Normal QQ plot: \(\beta\)}\\[2pt]
    \includegraphics[width=\textwidth]{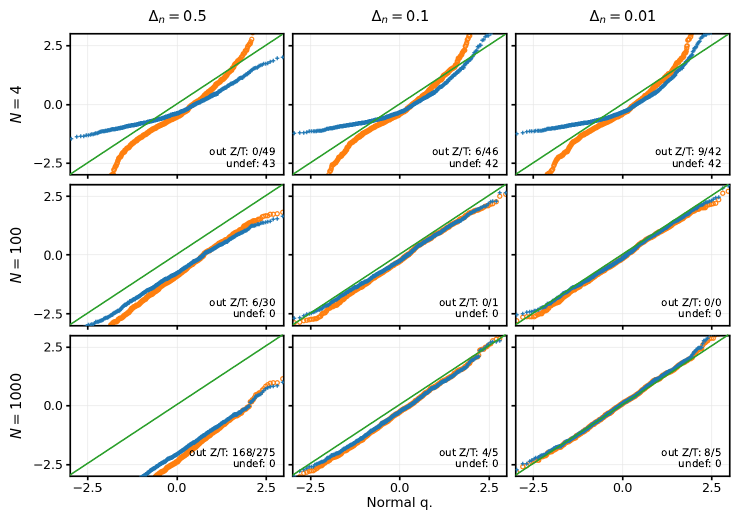}
\end{minipage}
\caption{Histograms and normal QQ plots for marginal standardized and
Studentized errors.  Panels~(a) and~(b) show paired histograms for
\(\alpha\) and \(\beta\), respectively.  The blue filled bars
show the limiting-covariance standardized errors
\(Z_{N,n}^{\alpha}\) and \(Z_{N,n}^{\beta}\), while the orange outlines
show the corresponding Studentized errors \(T_{N,n}^{\alpha}\) and
\(T_{N,n}^{\beta}\).  The green solid curve is
the standard normal density. The green solid vertical line marks zero, while
the blue dotted and orange dash-dotted lines mark the corresponding Monte Carlo
means. Panels~(c) and~(d) show paired normal QQ plots for the same two
standardizations: blue plus markers represent \(Z\), orange open circles
represent \(T\), and the green solid line is the identity line.  Within each
panel, rows correspond to \(N=4,100,1000\) and columns to
\(\Delta_n=0.5,0.1,0.01\).  The notation ``out Z/T'' records observations
outside the common displayed range, and ``undef'' records replications for
which \(\hat V_{N,n}\) is not defined.}
\label{fig:numerical-discrete-mesh-marginal}
\end{figure}

The marginal histograms and QQ plots in
\zcref[S]{fig:numerical-discrete-mesh-marginal} give a consistent picture: along
the diagonal combinations of \((N,\Delta_n)\), agreement with the Gaussian
reference improves, with the closest agreement at
\((N,\Delta_n)=(1000,0.01)\).  Coarser observation intervals retain a
systematic shift.

This shift is also visible directly at the \(\sqrt N\)-scale.  For
\(N=1000\), the Monte Carlo mean of
\(\sqrt N(\hat\theta_{N,n}-\theta_0)\) is approximately
\((2.28,-1.62)\) when \(\Delta_n=0.5\), whereas it is approximately
\((0.21,0.09)\) when \(\Delta_n=0.01\).  Thus the discretization bias that
remains visible at a coarse observation interval becomes much smaller on the
finer mesh, complementing the histogram and QQ-plot comparison above.

The paired histograms and QQ plots also show the effect of replacing the
limiting reference \(V\) by the empirical estimator \(\hat V_{N,n}\).  For
\(N=100\) and \(N=1000\), the Studentized distributions are close to the
limiting-covariance standardized ones, especially for the finer observation
intervals.  For \(N=4\), covariance estimation is much less stable.  In the
replications with \(\hat\alpha=0\), the plug-in information matrix
\(\hat I_{N,n}\) is singular, so \(\hat V_{N,n}\) is undefined.  These are
the \(43\), \(42\), and \(42\) ``undef'' cases for
\(\Delta_n=0.5,0.1,0.01\), respectively.  The boundary estimates are retained
in the Monte Carlo summaries.  No undefined covariance estimates occur for
\(N=100\) or \(N=1000\).

\subsubsection{Scatter and QQ plots for joint distribution}
\label{subsubsec:numerical-joint-distributions}

For the joint distribution, using the limiting covariance \(V\) as a reference,
and using the empirical sandwich covariance whenever it is defined, set
\begin{equation}\label{eq:numerical-whitened-error-vector}
\begin{aligned}
    Z_{N,n}
    &=
    V^{-1/2}\sqrt{N}
    \begin{pmatrix}
        \hat\alpha_{N,n}-\alpha_0\\
        \hat\beta_{N,n}-\beta_0
    \end{pmatrix},\\
    T_{N,n}
    &=
    \hat V_{N,n}^{-1/2}\sqrt{N}
    \begin{pmatrix}
        \hat\alpha_{N,n}-\alpha_0\\
        \hat\beta_{N,n}-\beta_0
    \end{pmatrix}.
\end{aligned}
\end{equation}
The first vector uses the limiting covariance, whereas the second is the joint
Studentized version.  Under the joint asymptotic normality and covariance
consistency results, both are expected to approach the bivariate standard
normal distribution \(\mathcal N(0,I_2)\).

For a visual examination of the joint distribution, we use paired scatter
plots of \(Z_{N,n}\) and \(T_{N,n}\).  The 95\% contour
of the reference distribution is
\[
    \bigl\{z\in\mathbb R^2: \|z\|^2=\chi^2_{2,0.95}\bigr\},
\]
and is shown by the green solid circle.

We also superimpose empirical 95\% Gaussian ellipses for the two point
clouds.  If \(\bar z\) and \(S_Z\) denote the empirical mean and covariance
matrix of the \(Z_{N,n}\) sample, and \(\bar t\) and \(S_T\) denote the
corresponding quantities for the defined \(T_{N,n}\) sample, the plotted
ellipses are
\[
    (x-\bar z)^\top S_Z^{-1}(x-\bar z)=\chi^2_{2,0.95},
    \qquad
    (x-\bar t)^\top S_T^{-1}(x-\bar t)=\chi^2_{2,0.95}.
\]
These are descriptive plug-in Gaussian probability ellipses for the two
empirical clouds, rather than confidence regions for their mean vectors.

To examine the radial distribution, define the squared Mahalanobis distances
\begin{equation}\label{eq:numerical-joint-chi-square}
    D^2_{N,n}(Z)=\|Z_{N,n}\|^2,
    \qquad
    D^2_{N,n}(T)=\|T_{N,n}\|^2.
\end{equation}
Under the corresponding bivariate standard normal reference, both squared
Mahalanobis distances have the \(\chi^2_2\) reference distribution.  The two
joint comparisons are shown in \zcref[S]{fig:numerical-discrete-mesh-joint}:
panel~(a) displays scatter plots of the standardized error vectors, and
panel~(b) compares
the empirical quantiles of \(D^2_{N,n}(Z)\) and \(D^2_{N,n}(T)\) with the
corresponding \(\chi^2_2\) quantiles.

\begin{figure}
\centering
\begin{minipage}[t]{0.49\textwidth}
    \centering
    \textbf{(a) Scatter plot}\\[2pt]
    \includegraphics[width=\textwidth]{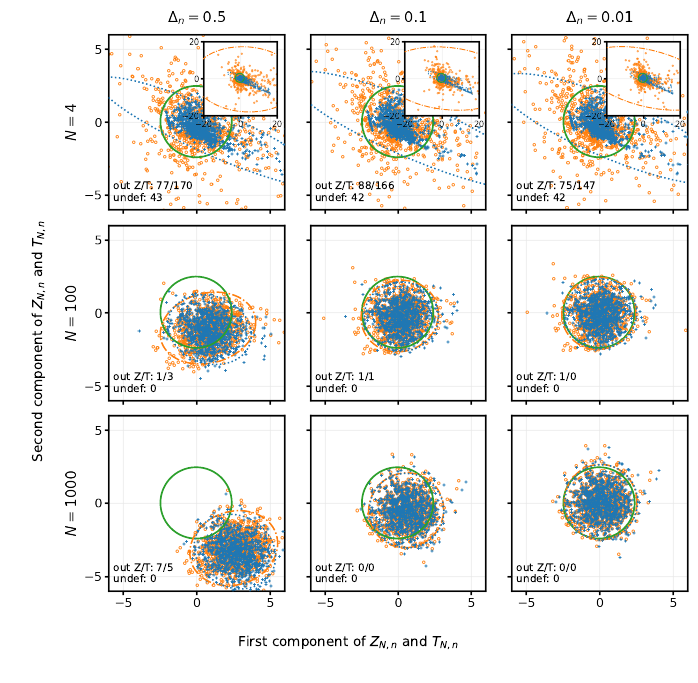}
\end{minipage}\hfill
\begin{minipage}[t]{0.49\textwidth}
    \centering
    \textbf{(b) QQ plot}\\[2pt]
    \includegraphics[width=\textwidth]{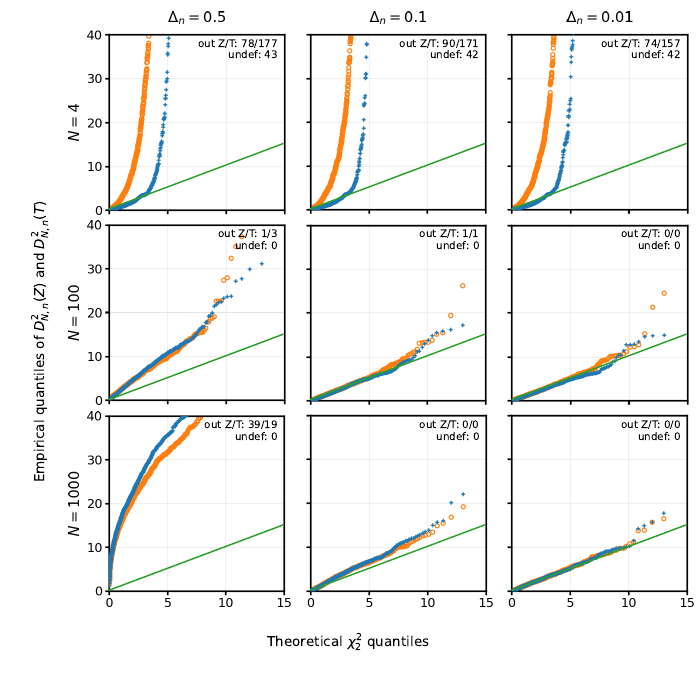}
\end{minipage}
\caption{Scatter and QQ plots for joint standardized and Studentized errors.
Rows correspond to \(N=4,100,1000\) and columns to
\(\Delta_n=0.5,0.1,0.01\).  Panel~(a) shows \(Z_{N,n}\) and \(T_{N,n}\)
defined in \eqref{eq:numerical-whitened-error-vector}: blue plus markers
represent \(Z_{N,n}\), orange open circles represent \(T_{N,n}\), the blue
dotted and orange dash-dotted ellipses are the corresponding empirical 95\%
Gaussian probability ellipses, and the green solid circle is the 95\% contour
of \(\mathcal N(0,I_2)\).  For \(N=4\), the insets show the same quantities
over \([-20,20]^2\). The main panels use \([-6,6]^2\).  Panel~(b) shows QQ
plots for \(D^2_{N,n}(Z)\) and \(D^2_{N,n}(T)\) in
\eqref{eq:numerical-joint-chi-square} against the \(\chi^2_2\) distribution.
Blue plus markers represent \(D^2_{N,n}(Z)\), orange open circles represent
\(D^2_{N,n}(T)\), and the green solid line is the identity line.  The
notation ``out Z/T'' records observations outside the common displayed range,
and ``undef'' records replications for which \(\hat V_{N,n}\) is not defined.}
\label{fig:numerical-discrete-mesh-joint}
\end{figure}

The joint scatter and QQ plots show the same dependence on the observation
interval as the marginal plots.  The closest agreement with the bivariate
Gaussian reference is seen for \((N,\Delta_n)=(1000,0.01)\), whereas coarser
observation intervals show clear departures in the location and radial
distribution of both the limiting-covariance standardized and Studentized
errors.  Taken together with the marginal plots and the scaled-bias
comparison above, these results are consistent with the role of the condition
\(\sqrt{N}\Delta_n\to0\) in the discrete asymptotic normality result.

\appendix
\section{Auxiliary lemmas}
\label{app:auxiliary-lemmas}

\label{app:parameter-differentiability}

We first collect the differentiability properties used in the main proof. Let \(C^{i,j}([0,T]\times\Theta)\) denote the space of functions that are \(i\)-times continuously differentiable in \(t\) and \(j\)-times continuously differentiable in \(\theta\).

\begin{lemma}
\label{lem:averaged-trajectory-time-regularity}
Let \(\Theta\) be compact and let \(\ell\ge0\). Suppose that \(H^\alpha\) and \(L^\beta\) are \(\ell\)-times continuously
differentiable in \(\theta\). Then \(\rhoave{t,\theta}\) is \(\ell\)-times continuously differentiable in \(\theta\). In addition, all its \(\theta\)-derivatives of order at most
\(\ell\) are uniformly bounded on \([0,T]\times\Theta\) and uniformly
Lipschitz continuous in \(t\), uniformly over \(\theta\in\Theta\).
\end{lemma}

\begin{proof}
By taking expectations in \eqref{eq:sample_path}, we have
\[
    \rhoave{t,\theta}
    =
    \rho_0+
    \int_0^t
        \mathcal{L}^{\theta}\bigl(\rhoave{s,\theta}\bigr)
    \,ds .
\]

Consider the Picard sequence
\[
    \rho_0(t,\theta)=\rho_0,
    \qquad
    \rho_{m+1}(t,\theta)
    =
    \rho_0+
    \int_0^t
        \mathcal{L}^{\theta}\bigl(\rho_m(s,\theta)\bigr)
    \,ds .
\]
Since the derivatives of \(\mathcal{L}^{\theta}\) up to order \(\ell\)
are uniformly bounded on \(\Theta\), differentiating
\begin{equation}
    \rho_{m+1}(t,\theta)-\rho_m(t,\theta)
    =
    \int_0^t
        \mathcal{L}^{\theta}
        \bigl(\rho_m(s,\theta)-\rho_{m-1}(s,\theta)\bigr)
    \,ds
\end{equation}
with respect to \(\theta\) gives
\begin{equation}
\begin{aligned}
\max_{0\leq j\leq\ell}
\sup_{\theta\in\Theta}
\norm{
    \partial_\theta^j
    \bigl(\rho_{m+1}(t,\theta)-\rho_m(t,\theta)\bigr)
}_{\mathrm{HS}}
&\leq
C\int_0^t
\max_{0\leq j\leq\ell}
\sup_{\theta\in\Theta}
\norm{
    \partial_\theta^j
    \bigl(\rho_m(s,\theta)-\rho_{m-1}(s,\theta)\bigr)
}_{\mathrm{HS}}
\,ds
\\
&\leq
\frac{(Ct)^m}{m!}
\max_{0\leq j\leq\ell}
\sup_{(s,\theta)\in[0,T]\times\Theta}
\norm{
    \partial_\theta^j
    \bigl(\rho_1(s,\theta)-\rho_0\bigr)
}_{\mathrm{HS}} .
\end{aligned}
\end{equation}
for some constant $C$. Hence \(\{\rho_m\}\) is a Cauchy sequence in
\(C^{0,\ell}([0,T]\times\Theta)\). 
Let \(\rho(t,\theta)\) denote this limit. Since
\(\mathcal{L}^{\theta}(\rho_m(t,\theta))\) also converges in
\(C^{0,\ell}([0,T]\times\Theta)\), passing to the limit in the integral equation
gives
\[
    \rho(t,\theta)
    =
    \rho_0+
    \int_0^t
        \mathcal{L}^{\theta}\bigl(\rho(s,\theta)\bigr)
    \,ds .
\]
By the uniqueness of the solution to the above integral equation, \(\rho(t,\theta)=\rhoave{t,\theta}\), and since
\(\mathcal{L}^{\theta}(\rho(t,\theta))\in C^{0,\ell}([0,T]\times\Theta)\),
 \(\rhoave{t,\theta}\) belongs to \(C^{1,\ell}([0,T]\times\Theta)\).
Hence its \(\theta\)-derivatives up to order \(\ell\) and their derivatives with respect to \(t\) are uniformly bounded. Consequently, the \(\theta\)-derivatives are uniformly Lipschitz continuous in \(t\), uniformly over \(\theta\in\Theta\).
\end{proof}

\begingroup
\begin{lemma}\label{lem:differentiation-stochastic-integral}
Let \(\ell\in\{1,2,3\}\), and suppose that \(\Theta^\circ\) is a bounded
Lipschitz domain with closure \(\Theta\). Let
\(Z=(Z^{(1)},\dots,Z^{(r)})\) be a continuous semimartingale on
\([0,T]\) of the form
\[
    dZ_t=a_t\,dt+dM_t,
\]
where \(a\) is progressively measurable and satisfies
\[
    \mathbb{E}\left[
        \left(\int_0^T \norm{a_t}_{\mathbb{R}^r}\,dt\right)^q
    \right]<\infty
\]
for every \(q\geq1\), and \(M\) is an \(r\)-dimensional continuous
martingale whose quadratic variation matrix satisfies
\[
    \operatorname{tr}\langle M\rangle_t \leq C_M t,
    \qquad 0\leq t\leq T,
\]
for some constant \(C_M\). Let
\(f:[0,T]\times\Theta\to\mathbb{R}^r\) be \(\ell\)-times continuously
differentiable in \(\theta\) on \(\Theta^\circ\), and suppose that all
\(\theta\)-derivatives of total order at most \(\ell\) extend continuously and
boundedly to \([0,T]\times\Theta\). Then, for every \(q\geq1\), the random
field
\[
    I(\theta):=\int_0^T f(t,\theta)\cdot dZ_t
\]
belongs to \(W^{\ell,q}(\Theta^\circ)\) almost surely, with weak derivatives
\[
    \partial_\theta^\nu I(\theta)
    =
    \int_0^T \partial_\theta^\nu f(t,\theta)\cdot dZ_t,
    \qquad |\nu|\leq \ell,
\]
where \(\nu\) is a multi-index. Moreover, if \(q>\dim\Theta\), then \(I\)
admits a modification in \(C^{\ell-1}(\Theta)\), and for \(|\nu|\leq \ell-1\)
the corresponding classical derivatives are given by the same stochastic
integrals.
\end{lemma}

\begin{proof}
For any bounded deterministic integrand \(g(t,\theta)\), the moment assumption
on \(a\), the Burkholder--Davis--Gundy inequality, and the quadratic-variation
bound give
\[
\mathbb{E}\left[
    \left|\int_0^T g(t,\theta)\cdot dZ_t\right|^q
\right]
\lesssim_q
\sup_{0\leq t\leq T}\norm{g(t,\theta)}_{\mathbb{R}^r}^q.
\]
Hence, for every multi-index \(\nu\) with \(|\nu|\leq \ell\),
\[
\mathbb{E}\left[
    \left\|
        \int_0^T \partial_\theta^\nu f(t,\theta)\cdot dZ_t
    \right\|_{L^q(\Theta^\circ)}^q
\right]<\infty.
\]
For a coordinate vector \(e_a\) and \(|\nu|\leq \ell-1\), the same estimate
applied to the difference quotient gives
\[
\begin{aligned}
&\mathbb{E}\left[
\left|
\frac{1}{\varepsilon}
\left\{
\int_0^T \partial_\theta^\nu f(t,\theta+\varepsilon e_a)\cdot dZ_t
-
\int_0^T \partial_\theta^\nu f(t,\theta)\cdot dZ_t
\right\}
-
\int_0^T
\partial_{\theta_a}\partial_\theta^\nu f(t,\theta)\cdot dZ_t
\right|^q
\right]
\\
&\qquad\lesssim_q
\sup_{0\leq t\leq T}
\norm{
\frac{\partial_\theta^\nu f(t,\theta+\varepsilon e_a)
      -\partial_\theta^\nu f(t,\theta)}{\varepsilon}
-
\partial_{\theta_a}\partial_\theta^\nu f(t,\theta)
}_{\mathbb{R}^r}^q,
\end{aligned}
\]
which tends to zero as \(\varepsilon\to0\), locally uniformly in
\(\theta\in\Theta^\circ\). Thus the displayed stochastic integrals are the
weak \(\theta\)-derivatives of \(I\) through order \(\ell\), and
\[
\mathbb{E}\bigl[\|I\|_{W^{\ell,q}(\Theta^\circ)}^q\bigr]<\infty.
\]
If \(q>\dim\Theta\), the Sobolev--Morrey embedding on the bounded Lipschitz
domain \(\Theta^\circ\) yields a \(C^{\ell-1}\)-modification on \(\Theta\). The
weak derivatives of order at most \(\ell-1\) then agree with the corresponding
classical derivatives, which proves the final assertion.
\end{proof}
\endgroup

\begin{lemma} \label{lem:StochIntVanishesGeneral}
    Let \(\Theta^\circ\) be a bounded Lipschitz domain in a Euclidean
    space with its closure \(\Theta\). Let
    \(f:[0,T]\times\Theta\to\mathbb{R}^r\) be a function. Assume that,
    for each \(t\in[0,T]\), the map \(\theta\mapsto f(t,\theta)\) is
    differentiable on \(\Theta^\circ\), and that \(f\) and
    \(\partial_\theta f\) are uniformly continuous in \((t,\theta)\) on their
    respective domains. Then
    \begin{equation}
        \int_0^T f(t,\theta)\, d\Wbar_t \convas 0,
    \end{equation}
    as \(N\to\infty\), uniformly in \(\theta\in\Theta\).
\end{lemma}

\begin{proof}
Choose an integer $m$ sufficiently large so that
$2m>\dim\Theta$. By the Sobolev--Morrey embedding theorem
\cite[Theorem~4.12]{adams2003sobolev}, applied on $\Theta^\circ$,
\begin{equation}
\mathbb{E}\qty[
\norm{
\int_0^T f(t,\theta)\,d\Wvec[1]_t
}_{C(\Theta)}^{2m}
]
\lesssim
\mathbb{E}\qty[
\norm{
\int_0^T f(t,\theta)\,d\Wvec[1]_t
}_{W^{1,2m}(\Theta^\circ)}^{2m}
].
\end{equation}

For $\ell=0,1$, by the Burkholder--Davis--Gundy inequality,
\begin{align*}
\mathbb{E}\qty[
\norm{
\int_0^T
\partial_\theta^\ell f(t,\theta)\,d\Wvec[1]_t
}_{L^{2m}(\Theta)}^{2m}
]
&=
\int_\Theta
\mathbb{E}\qty[
\abs{
\int_0^T
\partial_\theta^\ell f(t,\theta)\,d\Wvec[1]_t
}^{2m}
]\,d\theta \\
&\lesssim
\int_\Theta
\qty(
\int_0^T
\norm{\partial_\theta^\ell f(t,\theta)}_{\mathbb{R}^r}^2
\,dt
)^m
d\theta .
\end{align*}
Since $f$ and $\partial_\theta f$ are uniformly bounded on their
respective domains, the right-hand side is finite. Hence
\begin{equation}
\mathbb{E}\qty[
\norm{
\int_0^T f(t,\theta)\,d\Wvec[1]_t
}_{C(\Theta)}^{2m}
]
<\infty.
\end{equation}
 Moreover, for every $\theta\in\Theta$,
\begin{equation}
\mathbb{E}\qty[
\int_0^T f(t,\theta)\,d\Wvec[1]_t
]
=0 .
\end{equation}
Since $C(\Theta)$ is separable, the preceding integrability bound implies
that this $C(\Theta)$-valued random element is Bochner integrable with
mean zero. Therefore, the strong law of large numbers in a separable Banach
space \cite[Corollary~7.10]{ledoux1991probability} yields the desired
uniform convergence on $\Theta$.
\end{proof}

\begin{remark}
In the main text, the parameter space is written as
\(\Theta=\overline{\Theta^\circ}\), where \(\Theta^\circ\) is a bounded
convex open set. Since every bounded convex open set is a Lipschitz
domain, the assumptions of Lemma~\ref{lem:StochIntVanishesGeneral} are
satisfied.
\end{remark}

\begin{lemma}
\label{lem:StochIntVanishesCombined}
Let \(\Theta^\circ\) be a bounded Lipschitz domain in a Euclidean space
with its closure \(\Theta\). Let
\(f_n:[0,T]\times\Theta\to\mathbb{R}^r\), \(n\geq1\), be functions such that, for each \(t\in[0,T]\), the map
\(\theta\mapsto f_n(t,\theta)\) is differentiable on \(\Theta^\circ\).
Suppose that there exists a sequence \(a_n>0\) such that
\begin{equation}
    \sup_{(t,\theta)\in[0,T]\times\Theta}
    \norm{f_n(t,\theta)}_{\mathbb{R}^r}
    +
    \sup_{(t,\theta)\in[0,T]\times\Theta^\circ}
    \norm{\partial_\theta f_n(t,\theta)}_{\mathbb{R}^r}
    \leq a_n,
    \qquad n\geq1.
    \label{eq:combined-integrand-bound}
\end{equation}

\begin{itemize}
\item[(i)]
Let \(n=n(N)\) satisfy \(a_{n(N)}\to0\). Then
\begin{equation}
    \sqrt{N}
    \sup_{\theta\in\Theta}
    \abs{
        \int_0^T f_{n(N)}(t,\theta)\,d\Wbar_t
    }
    \convp 0 .
    \label{eq:combined-scaled-convergence}
\end{equation}

\item[(ii)]
Suppose that there exists an integer \(m>1\) such that
\(2m>\dim\Theta\) and
\begin{equation}
    \sum_{n=1}^\infty a_n^{2m}<\infty .
    \label{eq:combined-double-summability}
\end{equation}
Then
\begin{equation}
    \lim_{K\to\infty}
    \sup_{\substack{N\geq K\\ n\geq K}}
    \sup_{\theta\in\Theta}
    \abs{
        \int_0^T f_n(t,\theta)\,d\Wbar_t
    }
    =0
    \quad\text{a.s.}
    \label{eq:combined-double-convergence}
\end{equation}

\item[(iii)]
Suppose that there exists an integer \(m>1\) such that
\(4m>\dim\Theta\) and
\begin{equation}
    \sum_{n=1}^\infty a_n^{4m}<\infty .
    \label{eq:combined-empirical-summability}
\end{equation}
Then
\begin{equation}
    \lim_{K\to\infty}
    \sup_{\substack{N\geq K\\ n\geq K}}
    \frac{1}{N}\sum_{i=1}^N
    \sup_{\theta\in\Theta}
    \abs{
        \int_0^T f_n(t,\theta)\,d\Wvec[i]_t
    }^2
    =0
    \quad\text{a.s.}
    \label{eq:combined-empirical-convergence}
\end{equation}
\end{itemize}
The same conclusions hold for vector- or matrix-valued integrands with
finitely many components, by applying the above results componentwise.
\end{lemma}

\begin{proof}
For (i), put
\begin{equation}
    I_N(\theta)
    :=
    \int_0^T f_{n(N)}(t,\theta)\,d\Wbar_t .
\end{equation}
Taking an integer \(m\) such that \(2m>\dim\Theta\), the
Sobolev--Morrey embedding and the Burkholder--Davis--Gundy inequality,
together with \(\langle\Wbar\rangle_t=N^{-1}tI_r\), give
\begin{align}
    \mathbb{E}\qty[
        \norm{I_N}_{L^\infty(\Theta)}^{2m}
    ]
    &\lesssim
    \mathbb{E}\qty[
        \norm{I_N}_{W^{1,2m}(\Theta^\circ)}^{2m}
    ] \notag\\
    &\lesssim
    \frac{1}{N^m}
    \int_{\Theta^\circ}
    \qty(
        \int_0^T
        \left\{
            \norm{f_{n(N)}(t,\theta)}_{\mathbb{R}^r}^2
            +
            \norm{\partial_\theta f_{n(N)}(t,\theta)}_{\mathbb{R}^r}^2
        \right\}
        dt
    )^m
    d\theta
    \lesssim
    \frac{a_{n(N)}^{2m}}{N^m}.
    \label{eq:combined-scaled-moment}
\end{align}
Hence
\begin{equation}
    \mathbb{E}\qty[
        \qty(
            \sqrt{N}\norm{I_N}_{L^\infty(\Theta)}
        )^{2m}
    ]
    \lesssim
    a_{n(N)}^{2m}
    \longrightarrow0,
\end{equation}
which proves (i). For (ii), set
\begin{equation}
    I_{N,n}(\theta)
    :=
    \int_0^T f_n(t,\theta)\,d\Wbar_t .
\end{equation}
The same argument gives
\begin{equation}
    \mathbb{E}\qty[
        \norm{I_{N,n}}_{L^\infty(\Theta)}^{2m}
    ]
    \lesssim
    \frac{a_n^{2m}}{N^m}.
    \label{eq:combined-double-moment}
\end{equation}
Therefore, for every \(\varepsilon>0\), Markov's inequality yields
\begin{align}
    \sum_{N=1}^\infty\sum_{n=1}^\infty
    \mathbb{P}\qty(
        \sup_{\theta\in\Theta}
        \abs{I_{N,n}(\theta)}
        >
        \varepsilon
    )
    &\lesssim
    \frac{1}{\varepsilon^{2m}}
    \sum_{N=1}^\infty
    \frac{1}{N^m}
    \sum_{n=1}^\infty
    a_n^{2m}
    <\infty .
\end{align}
By the Borel--Cantelli lemma, for each \(\varepsilon>0\), almost surely
only finitely many pairs \((N,n)\) satisfy
\begin{equation}
    \sup_{\theta\in\Theta}
    \abs{I_{N,n}(\theta)}
    >
    \varepsilon .
\end{equation}
Taking \(\varepsilon\) over a countable sequence decreasing to zero proves
(ii). For (iii), put
\begin{equation}
    X_{i,n}
    :=
    \sup_{\theta\in\Theta}
    \abs{
        \int_0^T f_n(t,\theta)\,d\Wvec[i]_t
    }^2 .
\end{equation}
The same Sobolev--Morrey and Burkholder--Davis--Gundy estimates give
\begin{equation}
    \mathbb{E}\left[X_{i,n}^{2m}\right]
    \lesssim
    a_n^{4m},
    \qquad
    \mathbb{E}\left[X_{i,n}\right]
    \lesssim
    a_n^2 .
    \label{eq:combined-empirical-moments}
\end{equation}
Since \(X_{i,n}\), \(i\geq1\), are i.i.d., we have
\begin{equation}
    \mathbb{E}\qty[
        \abs{
            \frac{1}{N}\sum_{i=1}^N
            \left(
                X_{i,n}-\mathbb{E}[X_{i,n}]
            \right)
        }^{2m}
    ]
    \lesssim
    \frac{a_n^{4m}}{N^m}.
    \label{eq:combined-empirical-centered-moment}
\end{equation}
Thus, by Markov's inequality,
\begin{equation}
    \sum_{N=1}^\infty\sum_{n=1}^\infty
    \mathbb{P}\qty(
        \abs{
            \frac{1}{N}\sum_{i=1}^N
            \left(
                X_{i,n}-\mathbb{E}[X_{i,n}]
            \right)
        }
        >
        \varepsilon
    )
    <\infty .
\end{equation}
The Borel--Cantelli lemma therefore gives
\begin{equation}
    \lim_{K\to\infty}
    \sup_{\substack{N\geq K\\ n\geq K}}
    \abs{
        \frac{1}{N}\sum_{i=1}^N
        \left(
            X_{i,n}-\mathbb{E}[X_{i,n}]
        \right)
    }
    =0
    \quad\text{a.s.}
\end{equation}
Together with \(\mathbb{E}[X_{i,n}]\lesssim a_n^2\to0\), this proves (iii).
\end{proof}

\section*{Acknowledgements}
The authors thank Hiroki Nemoto for valuable discussions, helpful comments on earlier versions of the manuscript, and assistance in checking several calculations.

\bibliographystyle{amsplain}
\bibliography{references}
\end{document}